\documentclass[11pt,reqno]{amsart}
\usepackage{bm,amsmath,amsthm,amssymb,mathtools,verbatim,amsfonts,tikz-cd,mathrsfs,diagbox}
\usepackage{enumitem}
\usepackage{float}

\usepackage[hidelinks,colorlinks=true,linkcolor=blue, citecolor=black,linktocpage=true]{hyperref}
\usepackage{graphicx}
\usepackage[alphabetic]{amsrefs}
\usepackage{caption}
\usepackage{morefloats}
\usepackage[top=1.4in, bottom=1.2in, left=1.4in, right=1.4in,marginpar=1in]{geometry}
\usepackage{subfigure}
\usepackage{adjustbox}

\usetikzlibrary{matrix,arrows,decorations.pathmorphing,calligraphy}

\usepackage{chngcntr}
\counterwithin{figure}{section}

\newtheorem{thm}{Theorem}[section]
\newtheorem{prop}[thm]{Proposition}

\newtheorem{cor}[thm]{Corollary}
\newtheorem{lem}[thm]{Lemma}

\theoremstyle{definition}
\newtheorem{define}[thm]{Definition}

\theoremstyle{remark}
\newtheorem{rem}[thm]{Remark}
\newtheorem{example}[thm]{Example}

\newtheorem{question}[thm]{Question}

\newcommand{\ve}[1]{\boldsymbol{\mathbf{#1}}}

\newcommand{\Z}{\mathbb{Z}}

\renewcommand{\d}{\partial}
\renewcommand{\subset}{\subseteq}

\renewcommand{\tilde}{\widetilde}
\renewcommand{\bar}{\overline}

\newcommand{\iso}{\cong}

\DeclareMathOperator{\Ext}{{Ext}}

\DeclareMathOperator{\gr}{{gr}}

\DeclareMathOperator{\Hom}{{Hom}}

\DeclareMathOperator{\id}{{id}}

\DeclareMathOperator{\Span}{{Span}}

\newcommand{\lk}{\mathrm{lk}}

\newcommand{\bF}{\mathbb{F}}

\newcommand{\bI}{\mathbb{I}}

\newcommand{\bT}{\mathbb{T}}

\newcommand{\cA}{\mathcal{A}}
\newcommand{\cB}{\mathcal{B}}
\newcommand{\cC}{\mathcal{C}}

\newcommand{\cE}{\mathcal{E}}
\newcommand{\cF}{\mathcal{F}}

\newcommand{\cR}{\mathcal{R}}
\newcommand{\cS}{\mathcal{S}}

\newcommand{\cX}{\mathcal{X}}
\newcommand{\cY}{\mathcal{Y}}

\newcommand{\frd}{\mathfrak{d}}

\newcommand{\fro}{\mathfrak{o}}

\newcommand{\cCFL}{\mathcal{C\!F\!L}}
\newcommand{\cCFK}{\mathcal{C\hspace{-.5mm}F\hspace{-.3mm}K}}
\newcommand{\cHFL}{\mathcal{H\!F\! L}}

\newcommand{\CFL}{\mathit{CFL}}
\newcommand{\HFL}{\mathit{HFL}}

\newcommand{\xs}{\ve{x}}
\newcommand{\ys}{\ve{y}}
\newcommand{\zs}{\ve{z}}
\newcommand{\ws}{\ve{w}}

\newcommand{\as}{\ve{\alpha}}
\newcommand{\bs}{\ve{\beta}}

\renewcommand{\a}{\alpha}
\renewcommand{\b}{\beta}
\newcommand{\g}{\gamma}

\usepackage{leftidx}

\DeclareMathOperator{\Cone}{{Cone}}

\numberwithin{equation}{section}

\newcommand{\ar}{\mathrm{a.r.}}

\allowdisplaybreaks

\title{Link Floer homology, nonformality, and the Borromean rings}

\author{Kristen Hendricks}
\address{Department of Mathematics \\ Rutgers University New Brunswick \\ New Brunswick, NJ}
\email{kristen.hendricks@rutgers.edu}
\thanks{KH was partially supported by NSF grant DMS-2505573 and a Simons Travel Fellowship.}

\author{Matthew Stoffregen}
\address{Department of Mathematics \\ Michigan State University \\ East Lansing, MI}
\email{stoffre1@msu.edu}
\thanks {MS was partially supported by NSF grant DMS-2203828}

\author{Ian Zemke}
\address{Department of Mathematics\\University of Oregon\\  Eugene, OR}
\email{izemke@uoregen.edu}
\thanks{IZ was partially supported by a Sloan Fellowship, and NSF grants DMS-2204375 and DMS-2543889.}

\begin{document}
\maketitle

\begin{abstract} Our main result is a computation of  the full link Floer complex of the Borromean rings. The Borromean rings are well known to be an L-space link. We prove that the full link Floer complex is not a free-resolution of its homology, in contrast to the situation for L-space knots, plumbed L-space links, and two-component L-space links. We also describe the link involution of the Borromean rings, up to a minor ambiguity. The proof goes by way of studying some general techniques about $A_\infty$-deformations of L-space link Floer complexes. Our techniques give a simple criterion for when an L-space link must have formal link Floer complex, which reproves the existing formality results for one- and two-component L-space links.
\end{abstract}

\section{Introduction}

In this article, we compute the link Floer complex of the Borromean rings $\cB$. The Borromean rings form a 3-component link, and we will view the link Floer complex, denoted $\cCFL(L)$, as a free, finitely generated chain complex over  the six variable polynomial ring $\cR_3:=\bF[W_1,Z_1,W_2,Z_2,W_3,Z_3]$.

We now present our candidate, which we denote $\cC(\cB)$. We introduce the following notation. 
Write $B$ for the $1\times 1$ box complex over $\bF[W,Z]$, given by the diagram
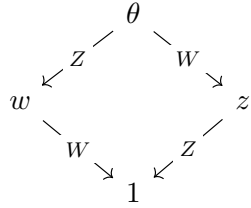
\[
\begin{tikzcd}[labels=description]& \theta \ar[dl, "Z"] \ar[dr, "W"]\\
w \ar[dr, "W"] &&  z \ar[dl, "Z"]\\
&1
\end{tikzcd}
\]
Write $B_i$ for the above complex over $\bF[W_i,Z_i]$, for $i\in \{1,2,3\}$.

The underlying vector space of $\cC$ is obtained by first taking the subquotient complex
\[
\cC_0\subset B_1\otimes_{\bF} B_2\otimes_{\bF} B_3
\] 
consisting of all generators except $111$ and $\theta\theta\theta$. Here we omit the tensor product symbol $\otimes$ when writing the generators, so $111$ denotes $1\otimes 1\otimes 1$. We can view $\cC_0$ as a free, finitely generated chain complex over $\cR_3$. 

We then adjoin two new generators $\ve{X}$ and $\ve{Y}$ to $\cC_0$ to form $\cC(B)$. If $\d_0$ denotes the differential inherited from $\cC_0$ being a subquotient complex of $B_1\otimes B_2\otimes B_3$, then
the differential on $\cC(B)$ is given by $\d_0+ \a$ where $\a$ is the following map
\begin{enumerate}
\item $\a(\ve{X})= Z_1 (w11)+Z_3 (11w)$.
\item $\a(\ve{Y})=Z_2 (1w1)+ Z_3 (11w)$.
\item $\a(\theta w w)=W_3 (1w1).$
\item $\a(w\theta w)=W_1 (11w)$.
\item $\a (ww\theta)=W_2 (w11)$.
\item $\a (\theta \theta w)=Z_1 (w1w)+W_3(\ve{X}+\ve{Y}).$
\item $\a(\theta w \theta)=Z_3 (1ww)+W_2 \ve{X}$.
\item $\a(w\theta\theta)=Z_2 (ww1)+W_1 \ve{Y}.$
\end{enumerate}

\begin{thm}\label{thm:main-computation} The complex $\cC(\cB)$, described above, is homotopy equivalent to $\cCFL(\cB)$ over $\cR_3$.
\end{thm}

We recall that $\cB$ is an L-space link, i.e. $S^3_{\Lambda}(L)$ is an L-space for all framings $\Lambda\gg 0$ \cite[Example 3.1]{LiuLSpaceLinks}.

The following question is asked in \cite{BLZLattice}:

\begin{question} Is the link Floer complex of an L-space link \emph{formal}, i.e. homotopy equivalent to a free-resolution of its homology?
\end{question}
\begin{rem}The above question is known to be true in some cases:
\begin{enumerate}
\item If $K$ is an L-space knot, then $\cCFK(K)$ is formal over $\bF[W,Z]$ by Ozsv\'{a}th and Szab\'{o}'s proof \cite[Proof of Theorem 1.2]{OSLens} that the knot Floer complex of an L-space knot is a staircase complex over $\bF[W,Z]$, which is easily seen to be a free-resolution of its homology.
\item If $L$ is a plumbed L-space link in a rational homology sphere (which includes the case that $L$ is an algebraic link \cite[Theorem 2]{GorskyNemethiAlgebraicLinks}), then $\cCFL(L)$ is formal over $\cR_n$ for $n=|L|$ by \cite[Theorem 1.2]{BLZLattice}. 
\item If $L$ is a two-component L-space link, then $\cCFL(L)$ is formal by \cite[Theorem 1.2]{CZZ}. 
\end{enumerate}
\end{rem}

None of the proofs of the above results naturally extend to non-plumbed links with more than two components. Indeed, we prove:

\begin{thm} The complex $\cCFL(\cB)$ is not formal.
\end{thm}
The high level overview of the proof is quite simple, and does not need the full power of our computation in Theorem~\ref{thm:main-computation}.
\begin{proof}[Proof sketch]
The homology group $\cHFL(\cB)$ is determined by the $H$-function of $\cB$, which takes the form of a map $H_\cB\colon \Z^3\to \Z^{\ge 0}$. By work of Gorsky and N\'{e}methi \cite[Theorems 1.5.1 and 1.5.2]{GorskyNemethiLattice},  $H_{\cB}$ is computable from the multivariable Alexander polynomial of the Borromean rings, to wit
\[
\Delta_{\cB}(t_1,t_2,t_3)=1-t_1-t_2-t_3+t_1t_2+t_1t_3+t_2t_3-t_1t_2t_3.
\] The function $H_{\cB}$ is well-known and can be computed from its Alexander polynomial. See \cite{GLLM_Borromean}*{Example~3.14}. From the $H$-function, we give a presentation of $\cHFL(\cB)$ as an $\cR_3$-module. Using Macaulay2 \cite{Mac2}, we computed a free-resolution over $\cR_3$. The result of this computation, as a reduced module, has $66$ generators over $\cR_3$. On the other hand, the Borromean rings form an alternating 3-component link, and therefore $\widehat{\HFL}(\cB)$ is determined by the Alexander polynomial $\Delta_{\cB}$ by Ozsv\'{a}th and Szab\'{o}'s work \cite[Section 4]{OSAlternating}. Using their result, the rank of $\widehat{\HFL}(\cB)$ is the sum of the absolute values of the coefficients of the Alexander polynomial, normalized as $(1-t_1)(1-t_2)(1-t_3) \Delta_{\cB}(t_1,t_2,t_3)$. This implies that $\widehat{\HFL}(\cB)$ has rank 64, which does not agree with the result from the free-resolution (i.e. 66), and therefore $\cCFL(\cB)$ cannot be homotopy equivalent to a free-resolution of its homology.
\end{proof}

The relationship between L-space links and their free-resolutions turns out to still be useful. We will prove below that the link Floer complex of every L-space link $L$ can be described as a certain kind of deformation of a free-resolution of its homology. Our strategy to compute $\cCFL(\cB)$ and prove Theorem~\ref{thm:main-computation} is to show that the free-resolution complex admits a unique such non-trivial deformation, thereby giving the answer.

We will also describe the link involution $\iota_\cB$ for the Borromean rings. The knot involution $\iota_K\colon \cCFK(K)\to \cCFK(K)$, for a knot $K$, was studied extensively in \cite{HMInvolutive}. We recall that the map $\iota_K$ is a homotopy equivalence which interchanges the actions of $W$ and $Z$. By convention, the map $\iota_K$ is referred to as the ``knot Floer involution'', though it does not in general square to the identity, even up to homotopy. The link involution is a small elaboration of this construction. See \cite[Section 15]{HHSZNaturality}. We note that the link involution involves performing a half twist on each link component, and therefore requires a choice of which direction to perform the half twist for each component. Therefore for the Borromean rings, there are \emph{a-priori} 8 link involutions, depending on these choices. In Section~\ref{sec:iota}, we show that there are  algebraically 8 distinct models for the link involution on $\cCFL(\cB)$, which yield mutually non-equivalent $\iota_L$-complexes. These 8 models are related by composing with the diffeomorphism for a Dehn twist around components of $\cB$, and therefore each of these 8 algebraic involutions is the correct involution for some choice of twisting directions. This is in contrast to the case of knots, where there is no known knot $K$ such that reversing the twisting direction produces a non-equivalent $\iota_K$-complex. 

We do not sort out which direction of twist produces which model of link involution. We conjecture that one could determine which twisting directions give which model of the link involution by working over $\Z$ using the ideas of \cite{AM:HFZ}. For the sake of brevity, we do not investigate this question further.

\subsection{Generalities about L-space links}

Although we show that the complex of the Borromean link is not a free-resolution of its homology, our work does give a simple condition which guarantees that the complex of an L-space link is a free-resolution of its homology:

\begin{thm}\label{thm:L-space-ext-groups-trivial} Suppose that $L$ is an L-space link in $S^3$ and that
\[
\Ext_{(-1,-1),\cR_n}^i(\cHFL(L),\cHFL(L))=0
\]
  for all $i>1$. In the above, $\Ext_{(-1,-1),\cR_n}^i$ denotes the subspace of the standard $\Ext$ group of $\cR_n$-modules which lies in $(\gr_{\ws},\gr_{\zs})$-grading $(-1,-1)$. Then $\cCFL(L)$ is homotopy equivalent over $\cR_n$ to a free-resolution if its homology.
\end{thm}

In Proposition~\ref{sec:L-space-one-two-components}, we show that Theorem~\ref{thm:L-space-ext-groups-trivial} can be used to prove the formality of L-space link complexes with one or two components, recovering the results of \cite{OSLens} and \cite{CZZ} respectively. 

In Section~\ref{sec:Ext-groups}, we show that the hypotheses of Theorem~\ref{thm:L-space-ext-groups-trivial} are not satisfied for the Borromean rings. We note that for a general L-space link $L$, the $\cR_n$-module structure on $\cHFL(L)$ is computable from the Alexander polynomial of $L$ and its sublinks by work of Gorsky and N\'{e}methi \cite{GorskyNemethiLattice}, and therefore the condition in the above theorem can be verified using a computer.

\subsection{Organization} This paper is organized as follows. In Section~\ref{sec:background} we review some structural facts about knot and link Floer homology. In Section~\ref{sec:algebra} we discuss some prerequisite homological algebra and prove several technical lemmas; we conclude by proving Theorem~\ref{thm:L-space-ext-groups-trivial} and showing that this theorem can be used to recover the formality of L-space links of one and two components. In Section~\ref{sec:borromean} we compute a free-resolution of the Borromean rings complex, and in Section~\ref{sec:deform} we find a deformation of this complex suitable to constructing the free link Floer complex, proving Theorem \ref{thm:main-computation}. Finally, in Section~\ref{sec:iota} we review the construction of the involutive knot and link involution and compute the set of link involutions for the Borromean rings up to an ambiguity of orientations. 

\subsection{Acknowledgments} We are grateful to I. Dai, J. Hom, C. Manolescu, and L. Truong for helpful discussions about the knot and link involution in Heegaard Floer homology (over many years). The last author also thanks M. Borodzik, B. Liu, D. Chen and H. Zhou for helpful discussions and their ideas about the L-space formality question. Portions of this work were carried out while the first and third authors were in residence at the 2026 Park City Mathematics Institute Summer Session; we are grateful for the conference's excellent hospitality.

\section{Background on knot and link Floer homology} \label{sec:background}

We now recall some basic material about knot and link Floer homology and establish our notation. We recall that knot Floer homology is an invariant of knots in 3-manifolds, introduced independently by Ozsv\'{a}th and Szab\'{o} \cite{OSKnots} and Rasmussen \cite{RasmussenKnots}. Link Floer homology is an extension of this theory to links due to Ozsv\'{a}th and Szab\'{o} \cite{OSLinks}.

In this article, we write $\cCFK(K)$ for the knot Floer invariant, which takes the form of a finitely generated free chain complex over the two variable polynomial ring $\cR:=\bF[W,Z]$. If $L\subset S^3$ is a link,  we write $\cCFL(L)$ for the link Floer complex, which is a free finitely generated chain complex over the ring $\cR_3:=\bF[W_1,Z_1,W_2,Z_2,W_3,Z_3]$.

For knots and links in $S^3$, the knot Floer complex $\cCFK(K)$ has two Maslov gradings, denoted $\gr_{\ws}$ and $\gr_{\zs}$, as well as an Alexander grading given by the formula
\[
A=\frac{1}{2}(\gr_{\ws}-\gr_{\zs}).
\]
We have the following grading shifts:
\[
(\gr_{\ws},\gr_{\zs})(\d)=(-1,-1), \quad (\gr_{\ws},\gr_{\zs})(W)=(-2,0), \quad (\gr_{\ws},\gr_{\zs})(Z)=(0,-2).
\]

If $L\subset S^3$ is a link, the invariant $\cCFL(L)$ also has a pair of Maslov gradings $\gr_{\ws}$ and $\gr_{\zs}$. The endomorphisms $\d$, $W_i$ and $Z_i$ have the same shifts as in the case of knots.  Additionally, $\cCFL(L)$ has an $n=|L|$ component Alexander grading, which takes values in the lattice
\[
\prod_{i=1}^n \left(\Z+\frac{\lk(K_i, L\setminus K_i)}{2}\right),
\]
where $K_1,\dots, K_n$ are the components of $K_i$. The variables $W_i$ and $Z_i$ have the following Alexander gradings:
\[
A_i(W_j)=-\delta_{ij}\quad \text{and} \quad A_i(Z_j)=\delta_{ij},
\]
where $\delta_{ij}$ denotes the Kronecker delta function. 

We have
\[
A_1+\cdots+A_n=\frac{1}{2}(\gr_{\ws}-\gr_{\zs}). 
\]

\section{Tools from homological algebra} \label{sec:algebra}

In this section, we describe some prerequisite results from homological algebra. 

\subsection{Type-$D$ and $A$ modules}

We will frequently use Lipshitz, Ozsv\'{a}th and Thurston's framework   of  type-$D$, $A$ and $DA$ bimodules \cite{LOTBordered}*{Chapter 2}. We refer the reader there for more details, though we present the definitions here. We focus on the case that we are working over an associative algebra $\cA$ over a ring $\ve{i}$ which is of characteristic 2.

A \emph{type-$D$ module} ${}^{\cA} X$ consists of a left $\ve{i}$ module $X$ equipped with an $\ve{i}$-linear map $\delta^1\colon X\to \cA\otimes_{\ve{i}} X$ which satisfies
\[
(\mu_2\otimes \bI_{X})\circ (\bI_{\cA}\otimes \delta^1)\circ \delta^1=0.
\]

A \emph{type-$D$ morphism} from ${}^{\cA} X$ to ${}^{\cA} Y$ consists of an $\ve{i}$-linear map
\[
f^1\colon X\to \cA\otimes Y.
\]
The hom space $\Hom_{\cA}(X,Y)$ admits a morphism differential
\[
d(f^1):=(\mu_2\otimes \bI_{X})\circ \left((\bI_{\cA}\otimes f^1)\circ \delta^1+(\bI_{\cA}\otimes \delta^1)\circ f^1 \right)
\]
which squares to 0. In particular, the category of type-$D$ modules over an associative algebra forms a $dg$-category.

A \emph{type-$A$ module} $X_{\cA}$ is the same thing as an $A_\infty$-module. We recall that this is the same as a right $\ve{i}$-module equipped with $\ve{i}$-linear maps
\[
m_{j+1}\colon X\otimes_{\ve{i}} \underbrace{\cA\otimes_{\ve{i}}\cdots \otimes_{\ve{i}} \cA}_{j}\to X
\]
for $j\ge 0$ which satisfy the following associativity relations for all $n$:
\[
\begin{split}
0=&\sum_{j=0}^n m_{n-j+1}(m_{j+1}(\ve{x},a_1,\dots, a_i), a_{i+1},\dots, a_n)
\\+
&\sum_{j=1}^{n-1} m_n(\xs,a_1,\dots, a_ja_{j+1},\dots, a_n).
\end{split}
\]
A morphism $f_*\colon M_{\cA}\to  N_{\cA}$ consists of a family of $\ve{i}$-linear maps
\[
f_{i+1}\colon M\otimes_{\ve{i}}\underbrace{\cA\otimes_{\ve{i}}\cdots \otimes_{\ve{i}} \cA}_i\to N
\]
ranging over $i\ge 1$.
The set of morphisms from $M_{\cA}$ to $N_{\cA}$ has a differential, given by
\[
\begin{split}
d(f_*)_{n+1}(\xs,a_1,\dots, a_n)=&\sum_{j=0}^n f_{n-j+1}(m_{j+1}(\xs,a_1,\dots, a_j), a_{j+1},\dots, a_n)\\
+&\sum_{j=0}^n m_{n-j+1}(f_{j+1}(\xs,a_1,\dots, a_j), a_{j+1},\dots, a_n)\\
+&\sum_{j=1}^{n-1} f_n(\xs,a_1,\dots, a_{j}a_{j+1},\dots, a_n).
\end{split}
\]
The category of type-$A$ modules forms a $dg$-category. 

We also recall that Lipshitz, Ozsv\'{a}th and Thurston describe a convenient model for the derived tensor product of a type-$A$ and a type-$D$ module, called the \emph{box tensor product} \cite{LOTBordered}*{Section~2.4}. If ${}^{\cA} Y$ and $X_{\cA}$ are type-$D$ and type-$A$ modules, then the underlying vector space of $X_{\cA}\boxtimes {}^{\cA}Y$  is given by $X\otimes_{\ve{i}} Y$. The differential on the box tensor product is given by the sum
\[
\sum_{i=0}^\infty (m_{i+1} \otimes \bI_{Y})\circ (\bI_X\otimes \delta^i).
\]
where $\delta^0=\bI_{Y}$ and $\delta^i\colon X\to \underbrace{\cA\otimes \cdots \otimes \cA}_i \otimes X$ is obtained by iterating $\delta^1$ $i$-times. Assuming suitable boundedness properties  on $X$ or $Y$, the infinite sum above involves only finitely many terms and determines a differential on $Y_{\cA}\boxtimes {}^{\cA}X$.

\subsection{Deformations of free-resolutions}

If $L$ is a link, write ${}_{\cR_n}\cHFL(L)$ for its link Floer homology, equipped with $A_\infty$-module actions so it is homotopy equivalent to ${}_{\cR_n}\cCFL(L)$; such object exists and is well-defined up to $A_\infty$-isomorphism using the homological perturbation lemma. Write ${}_{\cR_n}\tilde{\cHFL}(L)$ for the underlying $\cR_n$-module, with no higher actions.

\ If ${}_{\cR_n} M$ is an $\cR_n$-module, a \emph{free-resolution} of ${}_{\cR_n} M$ is a pair $({}^{\cR_n} F_*, \epsilon)$ where ${}^{\cR_n} F_*$ is a type-$D$ structure which admits a decomposition $F_*=\bigoplus_{i\ge 0} F_i$, whose structure map $\delta^1$ takes the form
\[
\begin{tikzcd}
\cdots \ar[r, "\delta^1"]& F_2  \ar[r, "\delta^1"] & F_1  \ar[r, "\delta^1"]& F_0.
\end{tikzcd}
\]
Furthermore, we assume that $\epsilon \colon {}_{\cR_n} [\cR_n]_{\cR_n}\boxtimes {}^{\cR_n}F_*\to {}_{\cR_n} M$ is a chain map of $\cR_n$ modules (i.e. a morphism of $A_\infty$-modules with only $\epsilon_1$ non-trivial) which is a quasi-isomorphism and is supported on $\cR_n\otimes F_0$. In the above definition, each $F_i$ is a vector space over $\bF$. Of course, each $F_i$ induces the free $\cR_n$ module $\cR_n\otimes_{\bF} F_i$ in the tensor product ${}_{\cR_n} [\cR_n]_{\cR_n} \boxtimes {}^{\cR_n} F_*$.  Sometimes it is helpful to conflate $\cR_n\boxtimes F_*$ and $F_*$.

We refer to the subscript $i$ in $F_i$ as the \emph{resolution grading}. Note that this will typically not be the same as the Maslov or homological grading because the ring $\cR_n$ is not concentrated in a single Maslov grading.

\begin{lem}
\label{lem:deformation} Let ${}_{\cR_n} M$ be an $A_\infty$-module with $m_1=0$. Let ${}_{\cR_n}\tilde{M}$ be the $\cR_n$-module obtained by forgetting the actions $m_j$ for $j>2$. Let ${}^{\cR_n} \cC$ be a type-$D$ structure so that ${}_{\cR_n}[\cR_n]_{\cR_n}\boxtimes {}^{\cR_n} \cC\simeq {}_{\cR_n} M$ and let ${}^{\cR_n} F_*$ be a free resolution of ${}_{\cR_n} \tilde{M}$. Then there is a type-$D$ endomorphism $\a^1$ of ${}^{\cR_n} F_*$ so that ${}^{\cR_n} \cC\simeq (F_*, \delta^1+\a^1)$. Furthermore, $\a^1$ may be taken to decrease the resolution grading by at least 2.
\end{lem} 
\begin{proof}
We consider the Koszul $DD$ bimodule ${}^{\cR_n}\Lambda^{\cR_n}$ whose underlying vector space is the exterior algebra on $2n$ generators, $\Lambda^*(\phi_1,\dots, \phi_n, \psi_1,\dots, \psi_n)$. We set $\delta^{1,1}(\phi_i)=W_i\otimes 1\otimes 1+1\otimes 1\otimes W_i$ and $\delta^{1,1}(\psi_i)=Z_i\otimes 1\otimes 1+1\otimes 1\otimes Z_i$. We extend $\delta^{1,1}$ to other generators of the exterior algebra via the Leibniz rule. We recall the well-known fact that there is a bounded homotopy equivalence
\[
{}^{\cR_n} \bI_{\cR_n} \simeq {}^{\cR_n} \Lambda^{\cR_n}\boxtimes {}_{\cR_n} [\cR_n]_{\cR_n}. 
\]
The above fact is a repackaging of the fact that
\[
{}^{\cR_n} \Lambda^{\cR_n}\boxtimes {}_{\cR_n}\tilde{M}
\]
is the Koszul resolution of ${}_{\cR_n} \tilde{M}$, and that the Koszul resolution is a free resolution. (See \cite{Weibel}*{Section~4.5 and Application 4.5.6} for background on the Koszul resolution). In the above, a \emph{bounded homotopy equivalence} is one where the maps $f_j$ appearing therein are non-zero for only finitely many $j$.

We observe that ${}^{\cR_n} \Lambda^{\cR_n}\boxtimes {}_{\cR_n}M$ is a deformation of ${}^{\cR_n} \Lambda^{\cR_n}\boxtimes {}_{\cR_n}\tilde{M}$ by some Maurer-Cartan element $\a^1$. The endomorphism $\a^1$ consists of all of the summands of the differential of the box tensor product which involve the action $m_{i+1}$ of ${}_{\cR_n} M$ for $i>1$.

 Therefore
\[
{}^{\cR_n} \cC\simeq {}^{\cR_n} \Lambda^{\cR_n} \boxtimes {}_{\cR_n} M \simeq ({}^{\cR_n} \Lambda^{\cR_n} \boxtimes {}_{\cR_n} \tilde{M},\delta^1+\a^1).
\]

The above argument shows that the free-resolution ${}^{\cR_n} \Lambda^{\cR_n}\boxtimes {}_{\cR_n}\tilde{M}$ has a deformation $\alpha^1$ as in the lemma statement.  We may translate $\alpha^1$ to an arbitrary free resolution ${}^{\cR_n} F_*$ as follows.  First, ${}_{\cR_n} \tilde{M}$ has a minimal free resolution ${}^{\cR_n} G_*$ so that any other free resolution decomposes as a direct sum ${}^{\cR_n}G_*\oplus {}^{\cR_n}T_*$, where ${}^{\cR_n}T_*$ is acyclic.  Moreover, ${}^{\cR_n}T_*$ will inherit its own resolution grading, together with $(\gr_{\ws},\gr_{\zs})$ gradings, and the differential on ${}^{\cR_n}T$ is degree $-1$ with respect to resolution grading, and $(-1,-1)$ in $(\gr_{\ws},\gr_{\zs})$.  Assume ${}^{\cR_n}G_*\oplus {}^{\cR_n}T_*$ is equipped with a perturbation $\alpha^1$.  The homological perturbation lemma, in the form of \cite[2.4]{Crainic} then applies.  To see this we set up some notation.  

Let $h^1$ be a nullhomotopy of the identity on ${}^{\cR_n} T_*$, which exists because ${}^{\cR_n}T_*$ is acyclic.  Let
\[A=(1-\alpha^1\circ h^1)^{-1}\circ\alpha^1.\]
The term $(1-\alpha^1 h^1)$ is invertible if we assume that $\alpha^1$ decreases resolution grading by at least $2$, since the homotopy $h^1$ can be chosen of degree $1$ in resolution grading, and $(1,1)$ in $(\gr_{\ws},\gr_{\zs})$ grading.  The existence of such homotopy follows the same lines as the proof that ${}_{\cR_n} \tilde{M}$ admits a minimal free resolution.  With those degrees on $\alpha^1,h^1$, it follows that applied to any given element of $F_*$, only finitely many terms in the formal series $\sum_{i=0} (\alpha^1 h^1)^i$ are non-zero.  Write $i^1,\pi^1$, respectively, for the inclusion and porjection of the ${}^{\cR_n}G_*$ summand.

Define $\alpha_G^1=\pi^1\circ \alpha^1 \circ i^1$.  Then \cite[2.4]{Crainic} applies to show that ${}^{\cR_n}G_*$, with deformation $\alpha_G^1$ is homotopy equivalent to ${}_{\cR_n}M$.  

That is, since the free resolution  ${}^{\cR_n} \Lambda^{\cR_n}\boxtimes {}_{\cR_n}\tilde{M}$ has a deformation as in the lemma statement, then the minimal free resolution has such deformation as well.  Moreover, if ${}^{\cR_n} F_*$ is any free resolution, then using the decomposition as ${}^{\cR_n}F_*={}^{\cR_n} G_*\oplus {}^{\cR_n}T$, we may give ${}^{\cR_n}F_*$ a deformation $\alpha^1$ by setting $\alpha^1=\alpha_G^1$ on ${}^{\cR_n}G_*$, and zero on the ${}^{\cR_n}T_*$ factor, completing the proof.
\end{proof}

\begin{cor}
\label{cor:deformation-link-Floer}  Let $L\subset S^3$ be a link.
Write ${}^{\cR_n}\cCFL(L)$ for the link Floer complex, and write ${}_{\cR_n} \tilde{\cHFL}(L)$ for the homology group $\cHFL(L)$, viewed as a type-$A$ module with $m_j=0$ for $j\neq 2$. Write
${}^{\cR_n}\tilde{\cCFL}(L)$ for a free-resolution of ${}_{\cR_n} \tilde{\cHFL}(L)$. The complex ${}^{\cR_n}\cCFL(L)$ is homotopy equivalent to a deformation of ${}^{\cR_n}\tilde{\cCFL}(L)$. That is, there is some endomorphism $\a^1$ of ${}^{\cR_n}\tilde{\cCFL}(L)$ so that ${}^{\cR_n}\cCFL(L)$ is homotopy equivalent to 
\[
(\tilde{\cCFL}(L),\delta^1+\a^1).
\] Furthermore the morphism $\a^1$ may be taken to lower the resolution grading by at least 2. 
\end{cor} 
\begin{proof}
Using the homological perturbation lemma, one may equip the homology group ${}_{\cR_n} \tilde{\cHFL}(L)$ with an $A_\infty$-module structure, denoted ${}_{\cR_n} \cHFL(L)$, which is homotopy equivalent to ${}_{\cR_n} [\cR_n]_{\cR_n} \boxtimes {}^{\cR_n} \cCFL(L)$. We therefore may apply Lemma~\ref{lem:deformation} with $M=\cHFL(L)$ and $\cC=\cCFL(L)$.
\end{proof}

 \subsection{Deformations and obstructions}

 We now recall some basic notions about deformations and obstruction theory.   The results we prove are adaptations to the category of chain complexes of the standard formalism of $\cA_k$-algebras and modules which are used to inductively construct $\cA_\infty$-algebras and related structures. See, e.g. \cite{Seidel_HMS_Quartic}*{Section~3}. 
 
 We consider in this section the classical notion of a chain complex of $\cR$-modules, where $\cR$ is an associative ring. This consists of a $\Z$-graded $\cR$ module $X_*=\bigoplus_{n\in \Z} X_n$, equipped with a differential $\d \colon X_n\to X_{n-1}$ which squares to 0, as shown below.
 \[
 \begin{tikzcd}
 \cdots\ar[r,"\d"]& X_n \ar[r, "\d"] &X_{n-1} \ar[r, "\d"] & X_{n-2} \ar[r, "\d"]& \cdots
 \end{tikzcd}.
 \]
 Note that this is more restrictive than a type-$D$ structure over $\cR$, which may not have grading structure as shown above. We assume, as elsewhere, that $\cR$ is a vector space over $\bF=\Z/2$. To disambiguate the two notions, we will refer to a chain complex in the above sense as a \emph{classical chain complex} of $\cR$-modules. If $X_*$ and $Y_*$ are classical chain complexes, we write $\Hom^n_{\cR}(X_*,Y_*)$ for the set of $\cR$-module maps which send $X_i$ to $X_{i+n}$. Recall that the $\Hom^*_{\cR}(X_*,Y_*)$ forms a chain complex with differential
 \[
 d(f)=\d\circ f+f\circ \d.
 \]

 \begin{define} Let $\cR$ be an algebra over $\bF=\Z/2$. 
 \begin{enumerate}
 \item  A \emph{$\cC_1$-complex} over $\cR$ is a classical chain complex.
 \item A \emph{$\cC_n$-complex} consists of a tuple  $\cX_{n}=(X_*,\d_1,\dots, \d_n)$ where  $\d_i\colon X_*\to X_*$, $i=1,\dots, n$, form a collection of maps such that 
\[
	\sum_{\substack{ i+j\le n+1}} \d_i\circ \d_j=0.
\]
We furthermore assume that $\d_i$ has grading $-i$. 
%\item A $C_\infty$-complex consists of an infinite tuple $(X_*,\d_1,\d_2,\dots)$ such that the above compatibility relation is satisfied for all $n$.
\end{enumerate}
\end{define}

If $\cX_{n}$ and $\cY_{n}$ are $\cC_n$-complexes with underlying $\cC_1$-complexes $X_*$ and $Y_*$, we define the space of morphisms $\Hom_{\cC_n}^i(\cX_{n},\cY_{n})$ to consist of  collections of morphisms $f=(f_0,\dots, f_{n-1})$ such that $f_j\in \Hom_{\cR}^{i-j}(X_*,Y_*)$. There is a boundary operator on $\Hom_{\cC_n}^*(\cX_{n},\cY_{n})$ which is given by the formula
\[
(d_n(f))_k=\sum_{\substack{i+j=k+1\\ n-1\ge i\ge 0\\ n\ge j\ge 1}} (f_i\circ \d_j+\d_j\circ f_i), 
\]
for $k\in \{0,\dots, n-1\}.$ The composition of two morphisms $f$ and $g$ is given by the equation
\[
(g\circ f)_k=\sum_{i+j=k} g_i \circ f_j
\]
where $k\in \{0,\dots, n-1\}$. 
We leave it to the reader to verify that $d_n^2=0$ and $d_n(f\circ g)=d_n(f)\circ g+f\circ d_n(g)$.  The above definition turns the category of $\cC_{n}$ -complexes into a $dg$-category. In particular, there is a natural notion of homotopy equivalence for $\cC_n$-complexes.

Given a $\cC_n$-complex $\cX_{n}$, we define the \emph{obstruction class}  $\fro(\cX_{n})$ to be the morphism
\[
\fro(\cX_{n})=\sum_{\substack{i+j=n+2\\ 1<i,j}} \d_i\circ \d_j\in \Hom^{-n-2}_{\cR}(X_*,X_*).
\]

\begin{lem}\label{lem:basic_deformation_lemma} Let $\cR$ be an algebra over $\Z/2$. 
\begin{enumerate}
\item Given a $\cC_n$-complex $\cX_{n}=(X_*,\d_1,\d_2,\dots, \d_n)$ over $\cR$ for $n\ge 1$, the obstruction $\fro(\cX_{n})$ satisfies $d(\fro(\cX_{n}))=0$. Furthermore, the induced class $\fro(\cX_{n})\in H_* \Hom^{-n-2}_{\cR}(X_*,X_*)$  vanishes if and only if $\cX_{n}$ extends to a $\cC_{n+1}$-complex. An extension $\d_{n+1}$ can be chosen to be any map in $\Hom^{-n-1}_{\cR}(X_*,X_*)$ which satisfies$d(\d_{n+1})=\fro(\cX_{n})$. 
\item If a $\cC_n$-complex $\cX_{n}$ extends to $\cC_{n+1}$-complexes $\cX_{n+1}$ and $\cX_{n+1}'$ with endomorphisms $\d_{n+1}$ and $\d_{n+1}'$, respectively, the extensions $\cX_{n+1}$ and $\cX_{n+1}'$ are homotopy equivalent if $[\d_{n+1}+\d_{n+1}']=0\in H_* \Hom^{-n-1}_{\cC_1}(X_*,X_*)$. \label{itm:equivalent-complexes}
\end{enumerate}
\end{lem}
\begin{proof} We begin with the proof that $\fro(\cX_{n})$ is a cycle. We compute that
\[
\begin{split}\d_1 \circ \fro(\cX_{n})&=\sum_{\substack{i+j=n+2\\i,j>1}} \d_1 \circ \d_i\circ \d_j\\
&=\sum_{\substack{k+m+j=n+3\\k,m,j>1}} \d_k \circ \d_m \circ \d_j+\sum_{\substack{i+j=n+2\\i,j>1}}   \d_i\circ\d_1\circ \d_j\\
&=\sum_{\substack{k+m+j=n+3\\k,m,j>1}} \d_k \circ \d_m \circ \d_j+\sum_{\substack{i+k+m=n+3\\ i,k,m>1}} \d_i \circ \d_k \circ \d_m+\sum_{\substack{i+j=n+2\\i,j>1}}   \d_i\circ \d_j\circ \d_1\\
&=\sum_{\substack{i+j=n+2\\i,j>1}}   \d_i\circ \d_j\circ \d_1\\
&=\fro(\cX_{n})\circ \d_1.
\end{split}
\]

 Next, if there exists $\partial_{n+1}$ as in the definition of an $\cC_{n+1}$-complex, then it follows from the definition that $d(\partial_{n+1}):=\d_1 \circ \d_{n+1}+\d_{n+1}\circ \d_1=\fro(\cX_{n})$, so the obstruction class is zero on homology.  Similarly, any morphism $\varphi\in \Hom^{-n-1}_{\cR}(X_*,X_*)$ with $d(\varphi)=\fro(\cX_{n})$ can be used to extend $\cX_{n}$ to a $\cC_{n+1}$-complex.  
 
 For the second claim, suppose that $\d_{n+1}$ and $\d_{n+1}'$ are two choices of endomorphisms which make $\cX_{n}$ into an $\cC_{n+1}$-complex. Write $\cX_{n+1}$ and $\cX_{n+1}'$ for the two extensions. Let $\eta\in \Hom_{\cR}^{-n}(X_*,X_*)$ be a map such that $d(\eta)=\d_{n+1}+\d_{n+1}'$. Then we can define a morphism $f\colon \cX_{n+1}\to \cX_{n+1}'$ by $f_0=\id$, $f_{n}=\eta$ and $f_i=0$ for $i\neq 0,n$. It is easy to see that $d_{n+1}(f)=0$. Furthermore the same definition gives a map from $\cX_{n+1}'$ to $\cX_{n+1}$ which we denote by $g$. Clearly $f\circ g=\id$ and $g\circ f=\id$ so $\cX_{n+1}$ and $\cX_{n+1}'$ are isomorphic.
\end{proof}

\begin{example} If $\cX_{1}$ is a $\cC_1$-complex, the obstruction $\fro(\cX_{1})$ vanishes, by convention. The second map $\d_2$ is required therefore only to satisfy $d_1(\d_2):=\d_1\circ \d_2+\d_2\circ \d_1=0$. If $\cX_{2}$ is a $\cC_2$-complex, then the obstruction class is $\d_2\circ \d_2$ and the map $\d_3$ is required to satisfy $d(\d_3):=\d_1\circ \d_3+\d_3\circ \d_1=\d_2\circ \d_2$. 
\end{example}

 \begin{rem} When $X_*$ is a free-resolution of a $\cR$-module $M$, the group $H_* \Hom_{\cR}^{-n-1}(X_*,X_*)$ coincides with $\Ext^{n+1}_{\cR}(M,M)$.
 \end{rem}

 We now discuss pushing the obstruction class along morphisms of $\cC_n$ complexes.
 
 \begin{lem}
 \label{lem:obstruction-class-homotopy-commutes} Suppose that $f\in \Hom_{\cC_n}^0(\cX_{n}, \cY_{n})$ is a degree 0 morphism of $\cC_n$-complexes with $d_n(f)=0$. Then $f_0\circ \fro(\cX_{n})\simeq \fro(\cY_{n})\circ f_0$. In particular, if $f$ is a homotopy equivalence of $\cC_n$ complexes, then $\cX_{n}$ is extendable if and only if $\cY_{n}$ is extendable.
 \end{lem}
 \begin{proof} If $f\colon \cX_{n}\to \cY_{n}$ is in $\Hom_{\cC_n}^0(\cX_{n},\cY_{n})$ and is a cycle, we can consider $\Cone(f)$, whose underlying complexes are $X_*\oplus Y_{*}[-1]$. This naturally forms a $\cC_n$-complex with differential 
 \[
 \d_i^{\Cone}=\d_i^{X}+\d_i^Y+f_{i-1}.
 \] The obstruction class $\fro(\Cone)$ may be identified as the sum $\fro(\cX_{n})+\fro(\cY_{n})+\fro(f)$ where $\fro(f)\in \Hom_{\cR}^{-n-1}(X_*,Y_*)$
 is the map
 \[
 \fro(f):=\sum_{\substack{1<i< n+1\\ 0<j<n \\j+i=n+1}} (\d_{i}' \circ f_j+f_j\circ \d_i),
 \]
 where $\d_i$ are the differentials from $\cX_n$ and $\d_i'$ are the differentials from $\cY_n$.
 
  By the first part of Lemma~\ref{lem:basic_deformation_lemma}, we know that
 \[
 \d_1^{\Cone}\circ \fro(\Cone)+\fro(\Cone)\circ \d_1^{\Cone}=0.
 \]
 Translating this to our present setting, this reads exactly that
 \begin{equation}
 f_0 \circ \fro(\cX_{n})+\fro(\cY_{n})\circ f_0+\d_1' \circ \fro(f)+\fro(f) \circ \d_1=0,
 \label{eq:f0-fro-homotopy-commute}
 \end{equation}
 which proves the claim.
 \end{proof}

 If $\cX_{n}$ is a $\cC_n$-complex which extends to a $\cC_{n+1}$ complex, write $\cE(\cX_{n})$ for the set of set of extensions of $\cX_{n}$, up to chain homotopy.

 \begin{lem}\label{lem:push-deformation} Suppose that $f\colon \cX_{n}\to \cY_{n}$ is an isomorphism of $\cC_n$-complexes which both extend to $\cC_{n+1}$ complexes. Then $f$ induces a canonical map $f_*\colon \cE(\cX_{n})\to \cE(\cY_{n})$. Furthermore, if $\d_{n+1}$ is an extension of $\cX_{n}$ and $f_*(\d_{n+1})$ is the corresponding extension of $\cY_{n}$, then $f$ extends to a homotopy equivalence from $\cX_{n+1}=(X_*,\d_1,\dots,, \d_{n+1})$ to $\cY_{n+1}=(Y_*,\d_1',\dots, f_*(\d_{n+1}))$.
 \end{lem}
 \begin{proof} Let $g\colon \cY_{n}\to \cX_{n}$ be an inverse to $f$. We derive the formula for $f_* (\d_{n+1})$. Write $\d_{n+1}'$ for $f_*( \d_{n+1})$. For $f$ to extend to a morphism of $\cC_{n+1}$-complexes, the desired relation is
 \begin{equation}
 0=\d_{n+1}'\circ f_0+f_0\circ \d_{n+1}+\d_1'\circ f_{n}+f_n \circ\d_1+\sum_{\substack{1<i< n+1\\ 0<j<n \\j+i=n+1}} (\d_{i}' \circ f_j+f_j\circ \d_i)
 \label{eq:equation-induced-map-on-extensions}
 \end{equation}
 We select $f_n=0$ and
 \[
 \d_{n+1}'=f_*(\d_{n+1})=f_0\circ \d_{n+1}\circ g_0+\sum_{\substack{1<i< n+1\\ 0<j<n \\j+i=n+1}} (\d_{i}' \circ f_j+f_j\circ \d_i)\circ g_0.
 \]
 Since $f_0\circ g_0=\id$, this choice clearly satisfies Equation~\eqref{eq:equation-induced-map-on-extensions}. Assuming that $\d_{n+1}'$ is an extension, it would follow that $f$ extends to an isomorphism from $\cX_{n+1}$ to $\cY_{n+1}$.

 It suffices therefore to show that $\d_{n+1}'$ extends $\cY_{n}$ to a $\cC_{n+1}$-complex. It suffices to show that $\d_{n+1}'$ is a nullhomotopy of $\fro(\cY_{n})$. We apply Equation~\eqref{eq:f0-fro-homotopy-commute} from the proof of Lemma~\ref{lem:obstruction-class-homotopy-commutes}, which shows that
 \[
 f_0\circ \fro(\cX_{n})+\fro(\cY_{n})\circ f_0+\d_1' \circ \fro(f)+\fro(f)\circ \d_1=0.
 \]
Composing with $g_0$ on the right, we obtain
\[
\fro(\cY_{n})=f_0\circ \fro(\cX_{n}) \circ g_0+ \d_1' \circ \fro(f)\circ g_0+\fro(f)\circ g_0\circ \d_1.
\]
Since $\d_{n+1}$ is a null-homotopy of $\fro(\cX_n)$, a null-homotopy of $\fro(\cY_n)$ is given by $f_0\circ \d_{n+1}\circ g_0+ \fro(f)\circ g_0$, which is our definition of $\d_{n+1}'$. 
  \end{proof}

  \begin{prop}
  \label{prop:unique-deformation} Suppose that $M$ is an $\cR_n=\bF[W_1,Z_1,\dots, W_n,Z_n]$ is a module which has a $(\gr_{\ws},\gr_{\zs})$ grading which takes values in $\Z\times \Z$. Let $F_*$ be a bounded free-resolution of $M$ (every $M$ admits such a resolution, e.g. the Koszul resolution).
  \begin{enumerate}
  \item  Suppose that $\Ext^i_{(-1,-1),\cR_n}(M,M)$ vanishes for all $i>1$. Then all $\cC_\infty$-deformations of $M$ are trivial, up to homotopy equivalence.  Here. $\Ext^i_{(-1,-1),\cR_n}(M,M)$ denotes the subgroup of the full $\Ext^i_{\cR_n}(M,M)$ group which lies in $(\gr_{\ws},\gr_{\zs})$-grading $(-1,-1)$.
  \item   Suppose that there is a number $m>1$ such that
 \[
 \Ext^i_{(-1,-1),\cR_n}(M,M)=\begin{cases}0 & \text{ if } 1<i, i\neq m\\
 \bF& \text{ if } i=m.
 \end{cases}
 \]
 Then, up to equivalence, there is a unique non-trivial $\cC_\infty$-deformation of
 $F_*$. 
 \end{enumerate}
 \end{prop}
 \begin{proof}
 We consider the first claim, where we assume that $\Ext^{i}_{(-1,-1), \cR_n}(M,M)$ vanishes for all $i>1$. We will show that $(F_*,\delta^1+\b_*)$ is homotopy equivalent to $(F_*,\delta^1)$ as a $\cC_\infty$-complex. 
 
 By truncating $(F_*,\delta^1+\b_*)$, we obtain a sequence of $\cC_j$-complexes $\tilde{\cF}_j$. Write $\cF_j$ for the truncation of $(F_*,\delta^1)$. Clearly $\cF_1$ and $\tilde{\cF}_1$ are identical. Note that $\tilde{\cF}_j$ is constant for $j\gg 0$ since $F_*$ is bounded. Proceeding by induction, we suppose that we have an isomorphism $\tilde{\cF}_j\simeq \cF_j$ for some $j$. Lemma~\ref{lem:push-deformation} gives an extension $\d_{j+1}$ of $\cF_j$ such that $\tilde{\cF}_j\simeq(F_*,\d_1, \d_{j+1})$. By Lemma~\ref{lem:basic_deformation_lemma}, the vanishing of the appropriate $\Ext$ groups implies that $(F_*,\d_1,\d_{j+1})$ is homotopy equivalent to $\cF_{j+1}= (F_*, \d_1,0)$. Proceeding by induction, we obtain $\tilde{\cF}_{N}\simeq \cF_{N}$ as $\cC_{N}$ complexes (for any $N$). Since $F_*$ is bounded, once $N$ is larger than the maximum resolution grading of $F_*$, the higher length relations must also be satisfied for these complexes and the maps between them, so $(F_*,\delta^1+\b_*)$ and $F_*$ are homotopy equivalent as $\cC_\infty$-complexes.

 The proof of the second claim is similar to the first. We first construct a $\cC_\infty$-deformation $\tilde{\cF}$ of $(F_*,\d_1)$. To do this, we let 
 \[
 \b_m\in \Hom_{(-1,-1),\cR_n}^{-m}(F_*,F_*)
 \] represent the unique non-zero class in  $\Ext^m_{(-1,-1),\cR_n}(M,M)$. We construct a $\cC_m$-complex as $\cF_m:=(F_*,\d_1,\b_m)$. The obstruction class $\fro(\cF_m)$ vanishes (up to chain homotopy) because it is an element of $\Ext^{m+2}_{(-1,-1),\cR_n}(M,M)$, which is zero by assumption, and therefore $\cF_m$ extends to a $\cC_{m+1}$-complex. Continuing in this manner, we may extend $\cF_m$ to a $\cC_\infty$-complex $\tilde{\cF}=(F_*,\d_1+\b_*)$ over $\cR_n$.
 
 The same argument as in the case where all the $\Ext$ groups vanish shows that every $\cC_\infty$ deformation of $F_*$ is homotopy equivalent to one of $\cF$ or $\tilde{\cF}$.

We now claim that $\cF$ and $\tilde{\cF}$ are not equivalent as $\cC_\infty$-complexes. It suffices to show that the truncations $\cF_m$ and $\tilde{\cF}_m$ are not-equivalent. Consider a morphism $f=(f_0,\dots, f_{m-1})$ between $\tilde{\cF}_m=(F_*,\d_1,\b_m)$ and $\cF_m=(F_*,\d_1,0)$. For this to be a chain map, we must have
 \[
 f_0\circ  \b_m=\d_1\circ f_{n-1}+f_{n-1}\circ \d_1. 
 \] 
If $f$ is a homotopy equivalence then $f_0$ must also be a homotopy equivalence, so $f_0\circ \b_m\simeq 0$ as a chain endomorphism of $(F_*,\d_1)$, which implies that $\b_m\simeq 0$, which is a contradiction.
 \end{proof}
 
 We now prove Theorem~\ref{thm:L-space-ext-groups-trivial} of the introduction, restated below:
 
 \begin{thm}\label{thm:Vanishing-Ext-Groups-Critereon} Suppose that $L$ is an L-space link in $S^3$ and that
 \[
 \Ext_{(-1,0,\dots,0),\cR_n}^i(\cHFL(L),\cHFL(L))=0
 \]
   for all $i>1$. In the above, $\Ext_{(-1,0,\dots ,0),\cR_n}^i$ denotes the subspace of the standard $\Ext$ group of $\cR_n$-modules which lies in $\gr_{\ws}$-grading $-1$ and Alexander grading $(0,\dots ,0)$. Then $\cCFL(L)$ is homotopy equivalent over $\cR_n$ to a free-resolution if its homology.
 \end{thm}
 \begin{proof} Write ${}^{\cR_n} \tilde{\cCFL}(L)$ for a free resolution of the $\cR_n$-module $\cHFL(L)$, and write ${}^{\cR_n} \cCFL(L)$ for the actual link Floer complex. Lemma~\ref{lem:deformation} implies that there is an endomorphism $\a^1$ of ${}^{\cR_n} \tilde{\cCFL}(L)$ which decreases the resolution grading by at least 2, and such that ${}^{\cR_n} \cCFL(L)\simeq (\tilde{\cCFL}(L), \delta^1+\a^1)$. We note that the morphism $\a^1$ has $\gr_{\ws}$-grading $-1$ and Alexander grading $(0,\dots, 0)$ by construction (since it has the same Maslov-Alexander grading as $\delta^1$). We can view $(\tilde{\cCFL}(L), \delta^1+\a^1)$ as a $\cC_\infty$-structure over the chain complex $(\tilde{\cCFL}(L),\delta^1)$. The first part of Proposition~\ref{prop:unique-deformation} implies that $(\tilde{\cCFL}(L), \delta^1+\a^1)$ is equivalent as a $\cC_\infty$-complex to $(\tilde{\cCFL}(L),\delta^1)$, and hence in particular ${}^{\cR_n} \cCFL(L)$ and ${}^{\cR_n} \tilde{\cCFL}(L)$ are homotopy equivalent as type-$D$ modules. 
 \end{proof}

\subsection{L-space links with one or two components}
\label{sec:L-space-one-two-components}
In this section, we sketch how Theorem~\ref{thm:Vanishing-Ext-Groups-Critereon} recovers the known formality results for one- and two-component L-space links, as in \cite{OSLens} and \cite{CZZ}. 

\begin{prop} If $L\subset S^3$ is a one- or two-component L-space link, then $\cCFL(L)$ is formal.
\end{prop}
\begin{proof} We begin with the case when $L$ has one component. Write $K$ for $L$ in this case. Write $A_s(K)\subset \cCFL(K)$ for the subspace in Alexander grading $s$. Since $K$ is an L-space knot,  the large surgery formula (see \cite{OSKnots}*{Theorem~4.1}) implies that $H_*(A_s(K))\iso \bF[U]$, where $U=WZ$. In particular, $\cHFL(L)$ is rank 0 or 1 in each $(\gr_{\ws},\gr_{\zs})$-bigrading. The free resolution of this complex is easily seen to be a staircase complex $\cS$. Since $\cS$ is concentrated in just two resolution gradings, we observe that $\Hom^i_{\cR}(\cS,\cS)$ vanishes if $i>1$, where $i$ denotes the algebraic grading shift of a morphism from $\cS$ to $\cS$. Hence the hypotheses of Theorem~\ref{thm:Vanishing-Ext-Groups-Critereon} are trivially satisfied. 

We now consider the case that $L\subset S^3$ is a two-component L-space link. In this case, we recall the candidate complex constructed in \cite{CZZ}*{Section~5.2}. The candidate complex was a DA-bimodule ${}_{\cR} \cX^{\cR}$, where $\cR=\bF[W,Z]$. The underlying type-$D$ structure (i.e. ignoring the type-$A$ actions) decomposed as 
\[
\cX^{\cR}=\bigoplus_{s\in \Z+\lk(L_1,L_2)/2} \cC_s^{\cR},
\]
 where $L_1,L_2$ are the components of $L$ and each $\cC_s$ is a staircase complex. There are only non-trivial actions $\delta_{i}^1$ for $i\in \{1,2,3\}$. Furthermore, $\delta_{i}^1$ respect the staircase grading in the sense that $\delta_1^1$ lowers the staircase grading, $\delta_2^1(a,-)$ preserves the staircase grading, and $\delta_3^1(a,b,-)$ increases the staircase grading by 1 (for all $a,b\in \cR$).

 Define  ${}^{\cR} \Lambda^{\cR}$ to be the DD-bimodule
\[
\begin{tikzcd}[labels=description, column sep=1.2cm, row sep=1.2cm]
& 
\ws
	\ar[dl, "W\otimes 1+1\otimes W"]
	\ar[dr, "Z\otimes 1+1\otimes Z"]\\
\xs
	\ar[dr, "Z\otimes 1+1\otimes Z"]
&& \ys 
	\ar[dl, "W\otimes 1+1\otimes W"]
\\
&\zs
\end{tikzcd}
\] The candidate complex ${}_{\cR} \cX^{\cR}$ has the property  that tensoring on the left with ${}^{\cR} \Lambda^{\cR}$ yields a free-resolution of the homology $\cHFL(L)$. (Here we view the type-DD bimodule ${}^{\cR}\Lambda^{\cR}\boxtimes {}_{\cR} \cX^{\cR}$ as a type-D module over $\cR\otimes \cR=\cR_2$.) See the proof of \cite{CZZ}*{Proposition~5.13}.   The resolution grading on ${}^{\cR}\Lambda^{\cR}\boxtimes {}^{\cR}\cX^{\cR}$ is the sum of the cube grading on $\Lambda$ and the staircase grading on $\cX$. Let us write ${}^{\cR} F_*^{\cR}$  for the tensor product ${}^{\cR}\Lambda^{\cR}\boxtimes {}^{\cR}\cX^{\cR}$.

We observe that ${}^{\cR} F_*^{\cR}$ is supported in resolution gradings $0,1,2,3$. We first consider a cycle $f\in \Hom^2_{(-1,0,\dots, 0),\cR_2}(F_*,F_*)$. We can write
\[
F_*=(\begin{tikzcd} F_3 \ar[r] & F_2 \ar[r]& F_1 \ar[r]& F_0  \end{tikzcd}).
\]
Since the subspaces of the staircases $\cC_s$ in resolution grading $i$ (modulo 2) lie in $\gr_{\ws}$ grading $i$ (modulo 2), and similarly for $\Lambda$, we observe that $F_i$ lies in $\gr_{\ws}$-grading $i$ (modulo 2). In particular, there are no $-1$ $\gr_{\ws}$-graded maps from $F_*$ to itself which drop the resolution grading by $2$. 

Finally, we consider a cycle $f\in \Hom^3_{(-1,0,\dots, 0),\cR_2}(F_*,F_*)$. Since there are only 3 resolution gradings supported by $F_*$, the only possibility is that $f$ consists of a map which sends $\ws\otimes \cX$ to $\zs\otimes \cX,$  and which drops the staircase grading of the $\cC_s$ complexes by 1. The map $f$ can be decomposed as a (potentially infinite) sum $f=\sum_{a} a\otimes f_a$, where the sum is over monomials $a\in \cR$ and $f_a$ is a map from $\cC_s$ to $\cC_{s-A(a)}$ which shifts the $\gr_{\ws}$ grading by $-\gr_{\ws}(a)-1$. The sum is potentially infinite, but on each $\cC_s$, only finitely many  $f_a$ are non-zero. 

 We observe that on $\Lambda\boxtimes \cX$, the map $a\otimes f_a$ is chain homotopic to the map $1\otimes \delta_2^1(a,-)\circ f_a$, and furthermore, since only finitely many of the $f_a$ are non-trivial on a given $\cC_s$, we can sum all of the chain homotopies to see that the sum $f=\sum_{a} a\otimes f_a$ is chain homotopic to $\sum_{a} 1\otimes \delta_2^1(a,-)\circ f_a$. In particular, we may assume, without loss of generality, that $f=1\otimes f_1$, where $f_1$ sends each $\ws\otimes \cC_s$ to  $\zs\otimes \cC_s$ and has $(\gr_{\ws},\gr_{\zs})$ bigrading $(-1,-1)$. By \cite{CZZ}*{Lemma~5.19}, any such map is null-homotopic as a map from $\cC_s$ to $\cC_s$. Such a null-homotopy gives a null-homotopy of $f$ when viewed as a map from $\Lambda\boxtimes \cX$ to $\Lambda\boxtimes \cX$, completing the proof. \end{proof}

\section{Borromean rings complex} \label{sec:borromean}

We consider the $H$ function of the Borromean rings. This is well known to be 
\[
H_{\cB}(\ve{s})=H_{U}(\ve{s}) +\delta_{\ve{s},\vec{0}},
\]
where $H_U$ is the $H$-function of the three component unlink and $\delta_{\ve{s},\ve{0}}$ is the Kronecker delta function. See \cite{GLLM_Borromean}*{Example~3.14}.

We will visualize $H_{\cB}$ as a sequence of 2-dimensional slices, shown below:
\[
\begin{array}{|c|ccccc||ccccc||ccccc|}
\hline
\hbox{\diagbox{$s_2$}{$s_1$}}
&-2&-1&0&1&2&-2&-1&0&1&2&-2&-1&0&1&2\\
\hline	\vdots
&3&2&1&1&1
&2&1&0&0&0
&2&1&0&0&0
\\
1&
3&2 &1&1&1 & 
2&1 &\boxed{0}&0&0 &
2&1 &0&0&0 
\\
	0
&3&2 &\boxed{1}&1&1
&2&\boxed{1} &1&\boxed{0}&0
&2&1 &\boxed{0}&0&0
\\
	-1
&4&3 &2&2&2
&3&2 &\boxed{1}&1&1
&3&2 &1&1&1
\\
	\vdots
&5&4 &3&3&3
&4&3 &2&2&2
&4&3 &2&2&2
\\
\hline
s_3&&&-1&&&&&0&&&&&1&&\\
\hline
\end{array}
\]

The corresponding module has six generators, which we denote by $x_0,x_1,x_2,x_3,x_4,x_5$. These are shown in the diagram below:
\[
\begin{array}{|c|ccccc||ccccc||ccccc|}
\hline
\hbox{\diagbox{$s_2$}{$s_1$}}
&-2&-1&0&1&2&-2&-1&0&1&2&-2&-1&0&1&2\\
\hline	\vdots
&3&2&1&1&1
&2&1&0&0&0
&2&1&0&0&0
\\
1&
3&2 &1&1&1 & 
2&1 &\boxed{x_2}&0&0 &
2&1 &0&0&0 
\\
	0
&3&2 &\boxed{x_5}&1&1
&2&\boxed{x_1} &1&\boxed{x_0}&0
&2&1 &\boxed{x_4}&0&0
\\
	-1
&4&3 &2&2&2
&3&2 &\boxed{x_3}&1&1
&3&2 &1&1&1
\\
	\vdots
&5&4 &3&3&3
&4&3 &2&2&2
&4&3 &2&2&2
\\
\hline
\end{array}
\]

\begin{lem} Let $x_1,\dots, x_6$ be the elements of $\cHFL(\cB)$ described above.
\begin{enumerate}
\item The elements $x_1,\dots, x_6$ form an $\cR_3$-module basis of the homology module $\cHFL(\cB)$.
\item The space of relations is generated over $\cR_3$ by the following elements:
\begin{enumerate}
\item 
$
U_i x_j=U_k x_j
$
for all $i$, $j$ and $k$.
\item $Z_3 x_5=W_3 x_4=Z_2  x_3=W_2 x_2=Z_1 x_1=W_1 x_0$.
\item $Z_3 x_0=Z_1 x_4$. 
\item $W_3 x_0=Z_1 x_5$.
\item $Z_2 x_0=Z_1 x_2$.
\item $W_2 x_0=Z_1 x_3$.
\item $W_2 x_1=W_1 x_3$.
\item $Z_3 x_1=W_1 x_4$.
\item $Z_2 x_1=W_1 x_2$.
\item $W_3 x_1=W_1 x_5$.
\item $Z_3 x_3=W_2 x_4$.
\item $W_3 x_3=W_2 x_5$.
\item $W_3 x_2=Z_2 x_5$.
\item $Z_3 x_2=Z_2 x_4$. 
\end{enumerate}
\end{enumerate}
\end{lem}
\begin{proof}
The generators and relations of $\cHFL(L)$ of an L-space link $L$ are described in \cite{BLZLattice}*{Lemma~7.1}.
Part one of the aforementioned lemma describes the unique minimal set of generators, and it is straightforward to identify this with $\{x_1,\dots, x_6\}$. The second part of the aforementioned lemma states that the generating set is spanned by the elements $(U_i+U_j) x_k$ as well as elements of the form $\a x_i+\b x_j$ where $\a$ and $\b$ are monomials such that $\a x_i$ and $\b x_j$ have the same Maslov and Alexander gradings. Furthermore, $\a$ and $\b$ may taken to be relatively prime, and furthermore \cite{BLZLattice}*{Lemma~7.1(2)} implies that we may take $w_U(\a)=w_U(\b)=0$, where $w_U(\a)$ is equal to $\sum_{i=1}^n \min(a_i,b_i)$ where $\a=\prod_{i=1}^n W_i^{a_i} Z_i^{b_i}$. Note that that by canceling common factors of $\a$ and $\b$, we may further assume that
\begin{equation}
\min (A_k(x_i), A_k(x_j))\le A_k(\a x_i)=A_k(\b x_j)\le \max (A_k(x_i), A_k(x_j))
\label{eq:inequality-Alexander-gradings}
\end{equation}
for all $k$.

We now show that the space of relations is spanned by the relations in the statement. By symmetry of $\cHFL(L)$, it suffices to show the claim for relations of the form $\a x_0+\b x_1$ and $\a x_2+\b x_0$.  For $x_0$ and $x_1$, Equation~\eqref{eq:inequality-Alexander-gradings} implies that the only possibilities for $A(\a x_0)=A(\b x_1)$ are of the form $(s_1,0,0)$ where $s_1\in \{-1,0,1\}$. The case that $s_1=-1$ may be excluded by assuming $w_U(\a)=0$, so our relation would be of the form $x_1= \b x_0$ which is not the case. Similarly we may exclude the case $s_1=1$. Therefore the only relation we need to add is of the form $Z_1 x_1=W_1 x_0$, which is in our list. The argument for relations of the form $\a x_2+\b x_0$ are similarly reduced to $Z_1 x_2=W_2 x_0$ and $W_2 x_2=W_1 x_0$, which are both relations in our list. The rest of the cases are handled similarly. 
\end{proof}

\subsection{The free-resolution complex}
\label{sec:free-resolution}

We used Macaulay2 \cite{Mac2} to compute a free-resolution of the module ${}_{\cR_3} \cHFL(\cB)$. This is produced using the following code.

\begin{verbatim}
R=ZZ/2[W_1,Z_1,W_2,Z_2,W_3,Z_3,
 	Degrees=>{{-2,0},{0,-2},{-2,0},{0,-2},{-2,0},{0,-2}}]
M_matrix=transpose matrix{
{W_1,Z_1,0,0,0,0},
{W_1,0,W_2,0,0,0},
{W_1,0,0,Z_2,0,0},
{W_1,0,0,0,W_3,0},
{W_1,0,0,0,0,Z_3},
{Z_3,0,0,0,Z_1,0},
{W_3,0,0,0,0,Z_1},
{Z_2,0,Z_1,0,0,0},
{W_2,0,0,Z_1,0,0},
{0,W_2,0,W_1,0,0},
{0,Z_3,0,0,W_1,0},
{0,Z_2,W_1,0,0,0},
{0,W_3,0,0,0,W_1},
{0,0,0,Z_3,W_2,0},
{0,0,0,W_3,0,W_2},
{0,0,W_3,0,0,Z_2},
{0,0,Z_3,0,Z_2,0},
{W_1*Z_1+W_2*Z_2,0,0,0,0,0},
{W_1*Z_1+W_3*Z_3,0,0,0,0,0},
{0,W_1*Z_1+W_2*Z_2,0,0,0,0},
{0,W_1*Z_1+W_3*Z_3,0,0,0,0},
{0,0,W_1*Z_1+W_2*Z_2,0,0,0},
{0,0,W_1*Z_1+W_3*Z_3,0,0,0},
{0,0,0,W_1*Z_1+W_2*Z_2,0,0},
{0,0,0,W_1*Z_1+W_3*Z_3,0,0},
{0,0,0,0,W_1*Z_1+W_2*Z_2,0},
{0,0,0,0,W_1*Z_1+W_3*Z_3,0},
{0,0,0,0,0,W_1*Z_1+W_2*Z_2},
{0,0,0,0,0,W_1*Z_1+W_3*Z_3}}
M_target=R^{{0,2},{2,0},{0,2},{2,0},{0,2},{2,0}}
f=map(M_target, , M_matrix)
M=cokernel f
F=res M
\end{verbatim}

Macaulay2 computes a free-resolution of $\cHFL(\cB)$ to be a complex of the following form
\[
F_*=(\begin{tikzcd}
F_0 & F_1\ar[l, "\d_1"] & F_2 \ar[l, "\d_1"] & F_3 \ar[l, "\d_1"] & F_4 \ar[l, "\d_1"] & F_5 \ar[l, "\d_1"]
\end{tikzcd})
\]

The complex $F_*$ is free of rank 66. The ranks of $F_0$, $F_1$, $F_2$, $F_3$, $F_4$ and $F_5$ are 6, 17, 21, 15, 6 and 1, respectively. We will enumerate the generators according to the ordering given by Macaulay2, though we will label them in a way which will help the subsequent steps. We will write 
\[
x_0,\dots, x_{42}, X_0, x_{43},\dots, x_{63}, Y_0
\]
for the generators of $F_0,\dots, F_5$, respectively (in the order produced by Macaulay2). It will be convenient to do the following change of basis:
\[
x_6\rightsquigarrow x_6+x_{11},\qquad x_{16}\rightsquigarrow x_{16}+x_{19}, \quad \text{and} \quad X_0\rightsquigarrow X_0+Z_1 x_{31}.
\]
 In Figure~\ref{fig:free-res-gens}, we display the generators of the free-resolution, their Alexander grading, and their differential.

\begin{figure}[h]
\[
\scalebox{.7}{$
\begin{array}{|c|c|c|c|}
\hline
x&\d(x)&A(x)&F_i\\
\hline
x_{0}&0 & ( 1,0 ,0 ) &F_0\\
x_{1}& 0& ( -1,0 ,0 ) &F_0\\
x_{2}&0 & ( 0,1 ,0 ) &F_0\\
x_{3}& 0& ( 0,-1 ,0 ) &F_0\\
x_{4}&0 & (0 ,0 ,1 ) &F_0\\
x_{5}& 0& ( 0,0 ,-1 ) &F_0\\
x_{6}& W_1x_0+Z_1x_1& ( 0, 0,0 ) &F_1\\
x_{7}&Z_2 x_1+W_1 x_2 & ( -1, 1,0 ) &F_1\\
x_{8}& W_2x_1+W_1x_3 & ( -1,-1 ,0 ) &F_1\\
x_{9}&Z_3x_1+W_1x_4 & ( -1, 0, 1) &F_1\\
x_{10}&W_3 x_1+W_1x_5 & (-1,0 ,-1)& F_1\\
x_{	11}&Z_1x_1+Z_3x_5 & ( 0, 0, 0) &F_1\\
x_{12}&  Z_2 x_0+Z_1x_2&( 1,1 ,0 ) &F_1\\
x_{13}& W_2x_0 +Z_1x_3&( 1, -1, 0) &F_1\\
x_{14}& Z_3x_0+Z_1x_4 &( 1,0 ,1 ) &F_1\\
x_{15}&W_3 x_0+Z_1x_5 &( 1, 0, -1) &F_1\\
x_{16}&W_2 x_2+Z_2x_3 &( 0, 0, 0) &F_1\\
x_{17}&Z_3x_3+W_2x_4 &( 0, -1, 1) &F_1\\
x_{18}&W_3x_3+W_2x_5 &( 0,-1 ,-1 ) &F_1\\
x_{19}&Z_2x_3+Z_3x_5 &( 0, 0,0 ) &F_1\\
x_{20}& Z_3 x_2+Z_2 x_4 &( 0, 1, 1) &F_1\\
x_{21}&W_3x_2+Z_2x_5 &( 0,1 ,-1 ) &F_1\\
x_{22}& W_2x_4+Z_3x_5 &( 0, 0,0 ) &F_1\\
x_{23}&Z_2x_6+Z_1x_7+W_1x_{12} &( 0, 1,0 ) &F_2\\
x_{24}&W_2x_6+Z_1x_8 +W_1x_{13} &( 0, -1, 0) &F_2\\
x_{25}& Z_3 x_6+Z_1x_9+W_1 x_{14} &( 0, 0, 1) &F_2\\
x_{26}& W_3 x_6+Z_1x_{10}+W_1x_{15} &( 0, 0, -1) &F_2\\
x_{27}&Z_3 x_8+W_2x_9+W_1x_{17} &( -1,-1 ,1 ) &F_2\\
x_{28}& W_3x_8+W_2x_{10}+W_1x_{18} &( -1,-1 ,-1 ) &F_2\\
x_{29}& Z_3x_{13}+W_2 x_{14}+Z_1x_{17} &( 1,-1 ,1 ) &F_2\\
x_{30}& W_3 x_{13}+W_2x_{15}+Z_1x_{18} &( 1, -1, -1) &F_2\\
x_{31}& W_2 x_7+Z_2 x_8+W_1x_{16} &( -1,0 ,0 ) &F_2\\
x_{32}&  Z_3 x_7+Z_2x_9+W_1x_{20}&( -1,1 ,1 ) &F_2\\
\hline
\end{array}
\quad
\begin{array}{|c|c|c|c|}
\hline
x&\d(x)&A(x)&F_i\\
\hline
x_{33}& W_3 x_7+Z_2x_{10}+W_1x_{21}&( -1, 1, -1) &F_2\\
x_{34}& W_2 x_{12}+Z_2 x_{13}+Z_1 x_{16} &( 1,0 ,0 ) &F_2\\
x_{35}&Z_3 x_{12}+Z_2x_{14}+Z_1x_{20} &( 1, 1,1 ) &F_2\\
x_{36}&W_2 x_{12}+Z_2 x_{15}+Z_1x_{21} &(1 ,1 ,-1 ) &F_2\\
x_{37}&Z_3 x_{16}+Z_2 x_{17}+W_2x_{20} &( 0, 0, 1) &F_2\\
x_{38}&W_3 x_{16}+Z_2x_{18}+W_2 x_{21} &( 0, 0, -1) &F_2\\
x_{39}& W_3 x_9+Z_3 x_{10}+W_1 x_{22} &( -1,0 ,0 ) &F_2\\
x_{40}&  W_3 x_{14}+Z_3 x_{15}+Z_1 x_{22}&( 1, 0, 0) &F_2\\
x_{41}& W_3 x_{17}+Z_3 x_{18}+W_2 x_{22}&( 0, -1, 0) &F_2\\
x_{42}& W_3 x_{20}+Z_3 x_{21}+Z_2 x_{22} &( 0,1 ,0 ) &F_2\\
X_0&Z_1Z_2x_8+Z_1 Z_3 x_{10}+Z_2Z_3x_{18}+(U_2+U_3)x_{11}+(U_1+U_3)x_{19} &( 0, 0, 0) &F_2\\
x_{43}&Z_3 x_{24}+W_2 x_{25}+Z_1 x_{27}+W_1 x_{29} &( 0, -1,1) &F_3\\
x_{44}& W_3 x_{24}+W_2 x_{26}+Z_1 x_{28}+W_1 x_{30} &(0 , -1,-1 ) &F_3\\
x_{45}&W_2x_{23}+Z_2 x_{24}+Z_1 x_{31}+W_1 x_{34} &( 0,0 ,0 ) &F_3\\
x_{46}& Z_3 x_{23}+Z_2 x_{25}+Z_1 x_{32}+W_1 x_{35}&( 0, 1, 1) &F_3\\
x_{47}& W_3 x_{23}+Z_2 x_{26}+Z_1 x_{33}+W_1 x_{36}&( 0,1 ,-1 ) &F_3\\
x_{48}& Z_2 x_{27}+Z_3 x_{31}+W_2 x_{32}+W_1 x_{37}&( -1, 0, 1) &F_3\\
x_{49}& Z_2 x_{28}+W_3 x_{31}+W_2 x_{33}+W_1 x_{38} &( -1,0 ,-1 ) &F_3\\
x_{50}&Z_2 x_{29}+Z_3 x_{34}+W_2 x_{35}+Z_1 x_{37} &( 1,0 ,1 ) &F_3\\
x_{51}& Z_2 x_{30}+W_3 x_{34}+W_2 x_{36}+Z_1 x_{38} &( 1,0 ,-1 ) &F_3\\
x_{52}& W_3 x_{25}+Z_3 x_{26}+Z_1 x_{39}+W_1 x_{40} &( 0,0 ,0 ) &F_3\\
x_{53}& W_3 x_{27}+Z_3 x_{28}+W_2 x_{39}+W_1 x_{41}&( -1,-1 ,0 ) &F_3\\
x_{54}& W_3 x_{29}+Z_3 x_{30}+W_2 x_{40}+Z_1 x_{41} &( 1, -1, 0) &F_3\\
x_{55}&W_3 x_{32}+Z_3 x_{33}+Z_2 x_{39}+W_1 x_{42} &( -1, 1,0 ) &F_3\\
x_{56}& W_3 x_{35}+Z_3 x_{36}+Z_2 x_{40}+Z_1 x_{42} &( 1,1 ,0 ) &F_3\\
x_{57}& W_3 x_{37}+Z_3 x_{38}+Z_2 x_{41}+W_2 x_{42} &( 0, 0,0 ) &F_3\\
x_{58}& Z_2 x_{43}+Z_3 x_{45}+W_2 x_{46}+Z_1 x_{48}+W_1 x_{50}&( 0, 0, 1) &F_4\\
x_{59}& Z_2 x_{44}+W_3 x_{45}+W_2 x_{47}+Z_1 x_{49}+W_1 x_{51}&( 0,0 ,-1 ) &F_4\\
x_{60}& W_3 x_{43}+Z_3 x_{44}+W_2 x_{52}+Z_1 x_{53}+W_1 x_{54}&( 0, -1,0 ) &F_4\\
x_{61}&W_3 x_{46}+Z_3 x_{47}+Z_2 x_{52}+Z_1 x_{55}+W_1 x_{56} &( 0,1 ,0 ) &F_4\\
x_{62}& W_3 x_{48}+Z_3 x_{49}+Z_2 x_{53}+W_2 x_{55}+W_1 x_{57} &( -1,0 ,0 ) &F_4\\
x_{63}&W_3 x_{50}+Z_3 x_{51}+Z_2 x_{54}+W_2 x_{56}+ Z_1 x_{57} &( 1, 0, 0) &F_4\\
Y_0&W_3x_{58} +Z_3x_{59}+Z_2x_{60}+W_2x_{61}+Z_1x_{62}+W_1x_{63} &( 0, 0,0 ) &F_5\\
\hline
\end{array}
$}
\]
\caption{The generators and differential of the free-resolution $F_*$. }
\label{fig:free-res-gens}
\end{figure}

\subsection{$\Ext$ groups of the Borromean rings}
\label{sec:Ext-groups}

The Macaulay2 code in Section~\ref{sec:free-resolution} can also be used to compute the $\Ext$ groups of the module $\cHFL(\cB)$. The code described therein gives $\cHFL(\cB)$ as a graded bimodule over the $(\gr_{\ws},\gr_{\zs})$-gradings. Macaulay2's conventions for graded resolutions are slightly different than ours, because its conventions are that the differential preserves the bigrading, whereas our convention is that it lowers the $(\gr_{\ws},\gr_{\zs})$ bigrading by $(-1,-1)$. Macaulay2's $(\gr_{\ws},\gr_{\zs})$-bigrading can be easily adjusted to match our convention by shifting the bigrading of $F_i$ up by $(i,i)$. In particular, to compute $\Ext_{\cR_3,(-1,-1)}^i(\cHFL(\cB),\cHFL(\cB))$ in our conventions  (that is, the subspace of $\Ext^i_{\cR_3}$ lying in $(\gr_{\ws},\gr_{\zs})$-grading $(-1,-1)$) we ask Macaulay2 to compute $\Ext_{\cR_3,(i-1,i-1)}^i(\cHFL(\cB),\cHFL(\cB))$ using its conventions. This is achieved by entering the following code:

\begin{verbatim}
E2=Ext^2(M,M)
E3=Ext^3(M,M)
E4=Ext^4(M,M)
E5=Ext^5(M,M)
basis({1,1},E2)
basis({2,2},E3)
basis({3,3},E4)
basis({4,4},E5)
\end{verbatim}
Macaulay2 proves the following:
\begin{lem}
\label{lem:ext-group-computation} We have 
\[
\Ext^i_{(-1,-1),\cR_3}(\cHFL(\cB),\cHFL(\cB))\iso\begin{cases} \bF& \text{ if } i=3\\
0& \text{ if } 1<i, i\neq 3
\end{cases}
\]
\end{lem}

\section{Deforming the free-resolution complex} \label{sec:deform}

In this section, we describe a deformation of the free-resolution complex $F_*$ computed in the last section.  Schematically, our deformation will take the following form:
\[
\begin{tikzcd}
F_0 & F_1\ar[l, "\d_1"] & F_2 \ar[l, "\d_1"] & F_3 \ar[l, "\d_1"] \ar[lll, "\a_3", bend right,swap]& F_4 \ar[l, "\d_1"] \ar[lll, "\b_3", bend right,swap]& F_5 \ar[l, "\d_1"] \ar[lll, "\g_3", bend right,swap]
\end{tikzcd}
\]
After defining $\a_3,\b_3$ and $\g_3$, we will prove that $\cCFL(\cB)$ is homotopy equivalent to $(F_*, \d_1+\d_3)$ where $\d_3:=\a_3+\b_3+\g_3$.

\subsection{The maps $\a_3$, $\b_3$ and $\g_3$}

We now define $\a_3$, $\b_3$ and $\g_3$. We define the map $\g_3$ to send $Y_0$ to $X_0$. Since $Y_0$ is the only generator of $F_5$, this completely defines $\g_3$. We now define $\b_3$ and $\a_3$ recursively so that $\d_3=\a_3+\b_3+\g_3$ satisfies the Mauer-Cartan equation
\[
\d_1 \circ \d_3+\d_3\circ \d_1+\d_3\circ \d_3=0.
\]

We solve for $\b_3$ and $\g_3$ by defining additional variables for the possible components of these maps, and then viewing the equations
\[
\d_1|_{F_2}\circ \g_3+\b_3\circ \d_1|_{F_5}=0\quad \text{and} \quad \a_3\circ \d_1|_{F_4}+\d_1|_{F_1}\circ \b_3=0
\]
as a large system of linear equations. This system turns out to be under-determined. Omitting straightforward linear algebra,  one may easily compute that the following maps are solutions
\[
\b_3=\begin{cases}x_{59}\mapsto &Z_1 x_{10}+W_3(x_{11}+x_{19})\\
x_{60}\mapsto& Z_3 x_{18}+W_2 x_{11}\\
x_{62}\mapsto & Z_2 x_8+W_1 x_{19}
\end{cases}
\qquad \text{and} \qquad \a_3=\begin{cases} x_{44}\mapsto & W_3 x_3\\
x_{49}\mapsto &W_1 x_5\\
x_{53}\mapsto &W_2 x_1.
\end{cases}
\]

We observe that the deformation $(F_*, \d_1+\d_3)$ is homotopy equivalent to the complex where we cancel $X_0$ and $Y_0$, because $X_0$ and $Y_0$ form an acyclic quotient complex. We write $\cC(\cB)$ for this cancelled complex, which  we summarize in Figure~\ref{fig:free-res-gens-canceled}. 

\begin{figure}[h]
\[
\scalebox{.7}{$
\begin{array}{|c|c|c|}
\hline
x&\d(x)&A(x)\\
\hline
x_{0}&0 & ( 1,0 ,0 ) \\
x_{1}& 0& ( -1,0 ,0 ) \\
x_{2}&0 & ( 0,1 ,0 ) \\
x_{3}& 0& ( 0,-1 ,0 ) \\
x_{4}&0 & (0 ,0 ,1 ) \\
x_{5}& 0& ( 0,0 ,-1 ) \\
x_{6}& W_1x_0+Z_1x_1& ( 0, 0,0 ) \\
x_{7}&Z_2 x_1+W_1 x_2 & ( -1, 1,0 ) \\
x_{8}& W_2x_1+W_1x_3 & ( -1,-1 ,0 ) \\
x_{9}&Z_3x_1+W_1x_4 & ( -1, 0, 1) \\
x_{10}&W_3 x_1+W_1x_5 & (-1,0 ,-1)\\
x_{	11}&Z_1x_1+Z_3x_5 & ( 0, 0, 0) \\
x_{12}&  Z_2 x_0+Z_1x_2&( 1,1 ,0 ) \\
x_{13}& W_2x_0 +Z_1x_3&( 1, -1, 0) \\
x_{14}& Z_3x_0+Z_1x_4 &( 1,0 ,1 ) \\
x_{15}&W_3 x_0+Z_1x_5 &( 1, 0, -1) \\
x_{16}&W_2 x_2+Z_2x_3 &( 0, 0, 0) \\
x_{17}&Z_3x_3+W_2x_4 &( 0, -1, 1) \\
x_{18}&W_3x_3+W_2x_5 &( 0,-1 ,-1 ) \\
x_{19}&Z_2x_3+Z_3x_5 &( 0, 0,0 ) \\
x_{20}& Z_3 x_2+Z_2 x_4 &( 0, 1, 1) \\
x_{21}&W_3x_2+Z_2x_5 &( 0,1 ,-1 ) \\
x_{22}& W_2x_4+Z_3x_5 &( 0, 0,0 ) \\
x_{23}&Z_2x_6+Z_1x_7+W_1x_{12} &( 0, 1,0 ) \\
x_{24}&W_2x_6+Z_1x_8 +W_1x_{13} &( 0, -1, 0) \\
x_{25}& Z_3 x_6+Z_1x_9+W_1 x_{14} &( 0, 0, 1) \\
x_{26}& W_3 x_6+Z_1x_{10}+W_1x_{15} &( 0, 0, -1) \\
x_{27}&Z_3 x_8+W_2x_9+W_1x_{17} &( -1,-1 ,1 ) \\
x_{28}& W_3x_8+W_2x_{10}+W_1x_{18} &( -1,-1 ,-1 ) \\
x_{29}& Z_3x_{13}+W_2 x_{14}+Z_1x_{17} &( 1,-1 ,1 ) \\
x_{30}& W_3 x_{13}+W_2x_{15}+Z_1x_{18} &( 1, -1, -1) \\
x_{31}& W_2 x_7+Z_2 x_8+W_1x_{16} &( -1,0 ,0 ) \\
\hline
\end{array}
\quad
\begin{array}{|c|c|c|}
\hline
x&\d(x)&A(x)\\
\hline
x_{32}&  Z_3 x_7+Z_2x_9+W_1x_{20}&( -1,1 ,1 ) \\
x_{33}& W_3 x_7+Z_2x_{10}+W_1x_{21}&( -1, 1, -1) \\
x_{34}& W_2 x_{12}+Z_2 x_{13}+Z_1 x_{16} &( 1,0 ,0 ) \\
x_{35}&Z_3 x_{12}+Z_2x_{14}+Z_1x_{20} &( 1, 1,1 ) \\
x_{36}&W_2 x_{12}+Z_2 x_{15}+Z_1x_{21} &(1 ,1 ,-1 ) \\
x_{37}&Z_3 x_{16}+Z_2 x_{17}+W_2x_{20} &( 0, 0, 1) \\
x_{38}&W_3 x_{16}+Z_2x_{18}+W_2 x_{21} &( 0, 0, -1) \\
x_{39}& W_3 x_9+Z_3 x_{10}+W_1 x_{22} &( -1,0 ,0 ) \\
x_{40}&  W_3 x_{14}+Z_3 x_{15}+Z_1 x_{22}&( 1, 0, 0) \\
x_{41}& W_3 x_{17}+Z_3 x_{18}+W_2 x_{22}&( 0, -1, 0) \\
x_{42}& W_3 x_{20}+Z_3 x_{21}+Z_2 x_{22} &( 0,1 ,0 ) \\
x_{43}&Z_3 x_{24}+W_2 x_{25}+Z_1 x_{27}+W_1 x_{29} &( 0, -1,1) \\
x_{44}& W_3 x_{24}+W_2 x_{26}+Z_1 x_{28}+W_1 x_{30}+W_3 x_3 &(0 , -1,-1 ) \\
x_{45}&W_2x_{23}+Z_2 x_{24}+Z_1 x_{31}+W_1 x_{34} &( 0,0 ,0 ) \\
x_{46}& Z_3 x_{23}+Z_2 x_{25}+Z_1 x_{32}+W_1 x_{35}&( 0, 1, 1) \\
x_{47}& W_3 x_{23}+Z_2 x_{26}+Z_1 x_{33}+W_1 x_{36}&( 0,1 ,-1 ) \\
x_{48}& Z_2 x_{27}+Z_3 x_{31}+W_2 x_{32}+W_1 x_{37}&( -1, 0, 1) \\
x_{49}& Z_2 x_{28}+W_3 x_{31}+W_2 x_{33}+W_1 x_{38} +W_1 x_5&( -1,0 ,-1 ) \\
x_{50}&Z_2 x_{29}+Z_3 x_{34}+W_2 x_{35}+Z_1 x_{37} &( 1,0 ,1 ) \\
x_{51}& Z_2 x_{30}+W_3 x_{34}+W_2 x_{36}+Z_1 x_{38} &( 1,0 ,-1 ) \\
x_{52}& W_3 x_{25}+Z_3 x_{26}+Z_1 x_{39}+W_1 x_{40} &( 0,0 ,0 ) \\
x_{53}& W_3 x_{27}+Z_3 x_{28}+W_2 x_{39}+W_1 x_{41}+W_2 x_1&( -1,-1 ,0 ) \\
x_{54}& W_3 x_{29}+Z_3 x_{30}+W_2 x_{40}+Z_1 x_{41} &( 1, -1, 0) \\
x_{55}&W_3 x_{32}+Z_3 x_{33}+Z_2 x_{39}+W_1 x_{42} &( -1, 1,0 ) \\
x_{56}& W_3 x_{35}+Z_3 x_{36}+Z_2 x_{40}+Z_1 x_{42} &( 1,1 ,0 ) \\
x_{57}& W_3 x_{37}+Z_3 x_{38}+Z_2 x_{41}+W_2 x_{42} &( 0, 0,0 ) \\
x_{58}& Z_2 x_{43}+Z_3 x_{45}+W_2 x_{46}+Z_1 x_{48}+W_1 x_{50}&( 0, 0, 1) \\
x_{59}& Z_2 x_{44}+W_3 x_{45}+W_2 x_{47}+Z_1 x_{49}+W_1 x_{51}+Z_1x_{10}+W_3(x_{11}+x_{19})&( 0,0 ,-1 ) \\
x_{60}& W_3 x_{43}+Z_3 x_{44}+W_2 x_{52}+Z_1 x_{53}+W_1 x_{54}+Z_3 x_{18}+W_2x_{11}&( 0, -1,0 ) \\
x_{61}&W_3 x_{46}+Z_3 x_{47}+Z_2 x_{52}+Z_1 x_{55}+W_1 x_{56} &( 0,1 ,0 ) \\
x_{62}& W_3 x_{48}+Z_3 x_{49}+Z_2 x_{53}+W_2 x_{55}+W_1 x_{57} +Z_2 x_8+W_1x_{19}&( -1,0 ,0 ) \\
x_{63}&W_3 x_{50}+Z_3 x_{51}+Z_2 x_{54}+W_2 x_{56}+ Z_1 x_{57} &( 1, 0, 0) \\
\hline
\end{array}
$}
\]
\caption{The complex $\cC(\cB)$ obtained from $(F_*,\d_1+\d_3)$ by canceling the acyclic quotient complex generated by $X_0$ and $Y_0$. }
\label{fig:free-res-gens-canceled}
\end{figure}

\subsection{Equivalence with $\cCFL(\cB)$}

In this section, we finish our proof of Theorem~\ref{thm:main-computation}, restated below:

\begin{thm} The link Floer complex of the Borromean rings $\cCFL(\cB)$ is homotopy equivalent to $\cC(\cB)$.
\end{thm}
\begin{proof} By Lemma~\ref{lem:deformation}, the link Floer complex $\cCFL(\cB)$ can be described by a deformation $(F_*,\d_1+\a)$ of the free-resolution complex where $\a=\a_2+\dots+\a_5$ where $\a_i$ drops the resolution grading by at least 2. The Ext groups $\Ext^i(\cHFL(\cB),\cHFL(\cB))$ satisfy the assumptions of Proposition~\ref{prop:unique-deformation} by Lemma~\ref{lem:ext-group-computation}, and therefore $F_*$ admits exactly one non-trivial deformation as a $\cC_\infty$-complex. Since $F_*/(W_1,Z_1,W_2,Z_2,W_3,Z_3)$ has homology of rank 66, whereas $\widehat{\CFL}(\cB)=\cCFL(\cB)/(W_1,Z_1,W_2,Z_2,W_3,Z_3)$ has homology of rank 64 by \cite{OSAlternating} because the graded Euler characteristic of  $\widehat{\CFL}(\cB)$ is (up to multiplication by a monomial) $(1-t_1)(1-t_2)(t-t_3)(1-t_1-t_2-t_3+t_1t_2+t_1t_3+t_2t_3-t_1t_2t_3)$ (when expanded, the sum of the absolute values of the coefficients of this polynomial is 64). In particular, $\cCFL(\cB)$ must be homotopy equivalent to the unique non-trivial deformation of $(\cF_*,\d_1)$, which is $(\cF_*,\d_1+\d_3)\simeq \cC(\cB)$. 
\end{proof}

\subsection{A compact description} \label{sec:compact}

Instead of the ``bag of generators'' description given in Figure~\ref{fig:free-res-gens-canceled}, we now offer a more compact description of the link Floer complex $\cCFL(\cB)$.

Write $B$ for the chain complex over $\bF[W,Z]$, written as
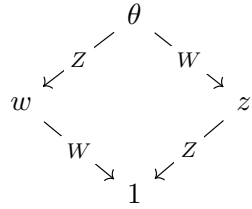
\[
\begin{tikzcd}[labels=description]& \theta \ar[dl, "Z"] \ar[dr, "W"]\\
w \ar[dr, "W"] &&  z \ar[dl, "Z"]\\
&1
\end{tikzcd}
\]

Then the Borromean link complex can be described by taking the the subquotient
\[
\cC_0\subset B_1\otimes_{\bF} B_2\otimes B_3
\] 
consisting of all generators except $111$ and $\theta\theta\theta$. We view the above as a type-$D$ structure over $\bF[W_1,Z_1,W_2,Z_2,W_3,Z_3]$ and will omit the tensor product symbol $\otimes$ when writing generators. We then adjoin two new generators $\ve{X}$ and $\ve{Y}$ to form a new vector space 
Then $\cC(\cB)=\cC_0\oplus \Span(\ve{X}, \ve{Y})$, with differential of the form $\d_0+ \a$. Here $\d_0$ is the differential of $\cC_0$, while $\a$ is the following map
\begin{enumerate}
\item $\a(\ve{X})= Z_1 w11+Z_3 11w$.
\item $\a(\ve{Y})=Z_2 1w1+ Z_3 11w$.
\item $\a(\theta w w)=W_3 1w1.$
\item $\a(w\theta w)=W_1 11w$.
\item $\a (ww\theta)=W_2 w11$.
\item $\a (\theta \theta w)=Z_1 w1w+W_3(\ve{X}+\ve{Y}).$
\item $\a(\theta w \theta)=Z_3 1ww+W_2 \ve{X}$.
\item $\a(w\theta\theta)=Z_2 ww1+W_1 \ve{Y}.$
\end{enumerate}

The translation between this version and the version in Figure~\ref{fig:free-res-gens-canceled} is given in Figure~\ref{fig:basis-correspondence}. 

\begin{figure}[h]
\[
\scalebox{.7}{$
\begin{array}{|c|c|}
\hline
\text{Basis 1}& \text{Basis 2}\\
\hline
x_{0}&z11  \\
x_{1}& w11 \\
x_{2}&1z1  \\
x_{3}& 1w1 \\
x_{4}&11z  \\
x_{5}& 11w \\
x_{6}& \theta11 \\
x_{7}&wz1 \\
x_{8}& ww1  \\
x_{9}&w1z  \\
x_{10}&w1w\\
x_{	11}&\ve{X} \\
x_{12}&  zz1 \\
x_{13}& zw1 \\
x_{14}& z1z  \\
x_{15}&z1w  \\
x_{16}&1\theta 1  \\
x_{17}&1wz  \\
x_{18}&1ww \\
x_{19}&\ve{Y} \\
x_{20}& 1zz  \\
x_{21}&1zw  \\
x_{22}& 1\theta\theta  \\
x_{23}&\theta z1 \\
x_{24}&\theta w 1  \\
x_{25}& \theta 1 z  \\
x_{26}& \theta 1 w  \\
x_{27}&wwz  \\
x_{28}& www  \\
x_{29}& zwz  \\
x_{30}& zww  \\
x_{31}& w\theta 1 \\
\hline
\end{array}
\quad
\begin{array}{|c|c|}
\hline
\text{Basis 1}&\text{Basis 2}\\
\hline
x_{32}&  wwz\\
x_{33}& wzw\\
x_{34}& z\theta 1 \\
x_{35}&zzz \\
x_{36}&zzw  \\
x_{37}&1\theta z  \\
x_{38}&1\theta w \\
x_{39}& w1\theta \\
x_{40}&  z1\theta \\
x_{41}& 1w\theta \\
x_{42}& 1z\theta \\
x_{43}&\theta w z  \\
x_{44}& \theta w w \\
x_{45}&\theta \theta 1 \\
x_{46}& \theta z z \\
x_{47}& \theta z w \\
x_{48}& w\theta z \\
x_{49}& w\theta w \\
x_{50}&z\theta z \\
x_{51}& z\theta w  \\
x_{52}& \theta 1 \theta  \\
x_{53}& ww\theta \\
x_{54}& zw\theta \\
x_{55}&wz\theta  \\
x_{56}& zz\theta \\
x_{57}& 1\theta\theta \\
x_{58}& \theta\theta z \\
x_{59}& \theta\theta w \\
x_{60}& \theta w \theta \\
x_{61}&\theta z \theta  \\
x_{62}&w\theta\theta \\
x_{63}&z\theta\theta \\
\hline
\end{array}
$}
\]
\caption{The correspondence between the two bases of $\cCFL(\cB)$. }
\label{fig:basis-correspondence}
\end{figure}

\pagebreak

\section{The link involution} \label{sec:iota}

In this section, we prove that up to equivalence, there are 8 different possible models of the link involution, each of which corresponds to a different choice of twisting direction in the construction of the link involution.

\subsection{The definition of the link involution}

We now briefly review the construction of the link involution, following the exposition in \cite{HHSZNaturality}*{Section~15}. Let $L\subset S^3$ be a link, and pick a choice of twisting directions $\frd$ for the components of $L$. Let $(\Sigma,\as,\bs,\ws,\zs)$ be a Heegaard diagram for $(S^3,L,\ws,\zs)$. One observes that $(-\Sigma,\bs,\as,\zs,\ws)$ is also a Heegaard diagram for $(S^3,L,\zs,\ws)$. One performs a half twist along $L$ in the twisting directions, which switches $\ws$ and $\zs$. Call the image of $(-\Sigma,\bs,\as,\zs,\ws)$ under this diffeomorphism map  $(\Sigma',\as',\bs',\ws,\zs)$. The link involution $\iota_{L,\frd}$ is the composition of the following maps
\[
\begin{tikzcd}[row sep=.7cm]\cCFL(\Sigma, \as,\bs,\ws,\zs)\ar[d, "\eta"] \\ \cCFL(-\Sigma,\bs, \as, \zs, \ws) \ar[d, "\phi_{\frd}"] \\ \cCFL(\Sigma',\bs',\as',\ws,\zs)\ar[d, "\Psi"] \\ \cCFL(\Sigma,\as,\bs,\ws,\zs). 
\end{tikzcd}
\]

In the above, $\eta$ is the canonical isomorphism which sends an intersection point $\xs\in \bT_{\a}\cap \bT_{\b}$ on $\Sigma$ to its image $\bar\xs \in \bT_{\b}\cap \bT_{\a}$ on $-\Sigma$. This map switches $W_i$ and $Z_i$. The map $\phi_{\frd}$ is the tautological map on intersection points when we apply the twisting diffeomorphism to the Heegaard surface. Finally the map $\Psi$ is the naturality map from \cite{JTNaturality}. 

If $\frd$ and $\frd'$ are choices of directions which are obtained by reversing the twisting direction of one component $K_i$, then the two involutions are related by the basepoint moving map along $K_i$, often called the \emph{Sarkar map}:
\[
\iota_{L,\frd'}\simeq (\id+\Phi_i\circ \Psi_i)\circ \iota_{L,\frd}.
\]
See \cite{SarkarMovingBasepoints} and \cite{ZemQuasi} for a computation and full discussion.
In the above, $\Phi_i$ is the formal derivative of the differential of $\cCFL(L)$ by the variable $W_i$; that is, it is obtained by writing the differential as an $n\times n$-matrix and taking the derivative with respect to $W_i$. The map $\Psi_i$ is the formal derivative with respect to $Z_i$. 

We will write $\xi_i$ for the Sarkar map on the $i$-th component:
\[
\xi_i=\id+\Phi_i \circ \Psi_i.
\]

The link involution satisfies
\[
\iota_{L,\frd}^2\simeq \xi_1\circ \xi_2\circ \xi_3=(\id+\Phi_1\circ \Psi_1)\circ (\id+\Phi_2\circ \Psi_2)\circ (\id+\Phi_3\circ \Psi_3).
\]
The above equation is proven by straightforward adaptation of the case of knots \cite{HMInvolutive}*{Section~6.2}. 

\subsection{The Sarkar maps}

We now compute the Sarkar map on the Borromean links complex, written in the notation of Section~\ref{sec:compact}. Recall that the Sarkar map $\xi$ corresponding to the full link is in fact a composition $\xi_1 \xi_2 \xi_3$ of the three Sarkar maps associated to each component. A straightforward computation shows that each of the maps $\xi_i$ for $i=1,2,3$ is derived from the $i$th Sarkar map on the complex $B_1 \otimes B_2 \otimes B_3$ as follows. Given a triple $abc \in \{1, z,w,\theta\}$, the $i$th Sarkar map on $abc \in B_1 \otimes B_2 \otimes B_3$ is obtained by adding the identity map to the map which returns the sum of all generators obtained by replacing an instance of $\theta$ in the $i$th position in $abc$ with $1$ in the $i$th position. This induces a map on the subquotient complex $\cC_0$, which concretely alters the map on the generators of $\cC_0$ by setting all instances of $111$ to zero. We then extend to $\cC_0 +\Span \{\ve{X}, \ve{Y}\}$ by letting $\xi_i$ be the identity on $\ve{X}$ and $\ve{Y}$. With this in mind, we present the full composition $\xi$ in Figure~\ref{fig:sarkar}.

\begin{figure}[h]
\[
\scalebox{.7}{$
\begin{array}{|c|c|}
\hline
\text{Basis 2}& \mathrm{Id} + \xi\\
\hline
\theta z 1 & 1z1\\
\theta w 1 & 1w1 \\
\theta 1 z & 11z \\
\theta 1 w & 11w \\
1 \theta z & 11z \\
1 \theta w & 11w \\
z \theta 1 & z11 \\
w \theta 1 & w11 \\
z 1 \theta & z11 \\
w 1 \theta & w11 \\
1 z \theta & 1z1 \\
1 w \theta & 1w1 \\
\theta zz & 1zz \\
\theta ww & 1ww \\
\theta wz & 1wz \\
\theta zw & 1zw \\
z \theta z & z1z \\
\hline
\end{array}
\quad
\begin{array}{|c|c|}
\hline
\text{Basis 2}& \mathrm{Id} + \xi\\
\hline
w \theta w & w1w \\
z \theta w & z1w \\
w \theta z & w1z\\
zz \theta & zz1 \\
ww \theta & ww1\\ 
wz \theta & wz1 \\
zw \theta & zw1 \\
\theta \theta 1 & 1\theta 1 + \theta 11 \\
\theta 1 \theta  & 11 \theta + \theta 11 \\
1 \theta \theta &  1 \theta 1 + \theta 11 \\
\theta \theta z & 1 \theta z + \theta 1 z + 11z \\
\theta \theta w &  1 \theta w + \theta 1 w + 11w\\
\theta z \theta & 1 z \theta + \theta z 1 + 1z1 \\
\theta w \theta  & 1 w \theta + \theta w 1 + 1w1 \\
z\theta \theta & z 1 \theta + z \theta 1 + z 11 \\
w\theta \theta & w 1 \theta + w \theta 1 + w 11 \\ 
\hline
\end{array}
$}
\]
\caption{The Sarkar map $\xi = \xi_1 \xi_2\xi_3$ is the sum of the identity map with the map shown in the table.}
\label{fig:sarkar}
\end{figure}

\subsection{The link involution}

We now compute the options for $\iota_L$. We first present our candidate involution. Choose $\alpha_1, \alpha_2 \in \{0,1\}$ such that $\alpha_1+\alpha_2=1$ modulo two, and similarly $\beta_1$ and $\beta_2$ and $\gamma_1$ and $\gamma_2$. We introduce some notation: given a generator $abc$ with $a,b,c \in \{1, w, z, \theta\}$, we let $\bar{a} \bar{b} \bar{c}$ be the triple given by applying the rule that $\bar{1}=1$, $\bar{z}=w$, $\bar{w}=z$, and $\bar{\theta}=\theta$. We call this map the \emph{bar map}.

Consider the map shown in Figure~\ref{fig:iotaL}.

\begin{figure}[h]
\[
\scalebox{.7}{$
\begin{array}{|c|c|}
\hline
\text{Basis 2}& \iota_L\\
\hline
abc \text{ if } a,b,c \in \{1,z,w \} & \bar{a}\bar{b}\bar{c}  \\
\theta 1 1 & \theta 1 1 \\
1 \theta 1 & 1 \theta 1 \\
1 1 \theta & 1 1 \theta \\
\ve{X} & \ve{X} + \theta 11 + 11\theta  \\
\ve{Y} & \ve{Y} + 1 \theta 1 + 11 \theta \\
\theta z 1 & \theta w 1 + \alpha_1 1w1 \\
\theta w 1 & \theta z 1 + \alpha_2 1z1 \\
\theta 1 z & \theta 1 w + \alpha_2 11w \\
\theta 1 w & \theta 1 z + \alpha_1 11z \\
1 \theta z & 1 \theta w + \beta_1 11w \\
1 \theta w & 1 \theta z + \beta_2 11z \\
z \theta 1 & w\theta 1 + \beta_2 w11 \\
w \theta 1 & z \theta 1 + \beta_1 z11 \\
z 1 \theta & w 1 \theta + \gamma_1 w11 \\
w 1 \theta & z 1 \theta + \gamma_2 z11 \\
1 z \theta & 1 w \theta + \gamma_2 1w1 \\
1 w \theta & 1 z \theta + \gamma_1 1z1 \\
\theta zz & \theta ww + \alpha_2 1 ww \\
\theta ww & \theta zz + \alpha_1 1 zz \\
\hline
\end{array}
\quad
\begin{array}{|c|c|}
\hline
\text{Basis 2}& \iota_L\\
\hline
\theta wz & \theta zw + \alpha_2 1zw \\
\theta zw & \theta wz + \alpha_1 wz \\
z \theta z & w \theta w + \beta_2 w1w \\
w \theta w & z \theta z + \beta_1 z1z \\
z \theta w & w \theta z + \beta_2 w1z \\
w \theta z & z \theta w + \beta_1 z1w\\
zz \theta & ww\theta + \gamma_2 ww1 \\
ww \theta & zz \theta + \gamma_1 zz1\\ 
wz \theta & zw \theta + \gamma_2 zw1 \\
zw \theta & wz \theta + \gamma_1 wz1 \\
\theta \theta 1 & \theta \theta 1 + \alpha_2 1 \theta 1 + \beta_1 \theta 1 1 + \ve{X} + \ve{Y} \\
\theta 1 \theta  & \theta 1 \theta + \alpha_1 11 \theta + \gamma_2 \theta 11 + \ve{X} \\
1 \theta \theta &  1 \theta \theta + \beta_2 11 \theta + \gamma_1 1\theta 1 + \ve{Y} \\
\theta \theta z & \theta \theta w + \alpha_2 1 \theta w + \beta_1 \theta 1 w + 11 w \\
\theta \theta w & \theta \theta z + \alpha_1 1 \theta z + \beta_2 \theta 1 z + (\alpha_1 \beta_1 + \alpha_2 \beta_2) 11z \\
\theta z \theta & \theta w \theta + \gamma_2 \theta w 1 + \alpha_1 1w\theta + 1w1 \\
\theta w \theta  & \theta z \theta + \gamma_1 \theta z 1 + \alpha_2 1z\theta + (\alpha_1 \gamma_1 + \alpha_2\gamma_2) 1z1 \\
z\theta \theta & w\theta \theta + \beta_2 w1\theta + \gamma_1 w \theta 1 + w11 \\
w\theta \theta & z \theta \theta + \beta_1 z1\theta + \gamma_2 z \theta 1 + (\beta_1\gamma_1 + \beta_2\gamma_2)z11 \\ 
\hline
\end{array}
$}
\]
\caption{The eight possibilities for $\iota_L$. }
\label{fig:iotaL}
\end{figure}

\begin{thm} \label{thm:iota} Up to change of basis, the link involution of the Borromean rings is of the form shown in Figure~\ref{fig:iotaL}. Similarly up to change of basis, exchanging the values of $\alpha_1$ and $\alpha_2$ corresponds to reversing the orientation of the first component; exchanging the values of $\beta_1$ and $\beta_2$ corresponds to reversing the orientation of the second component; and exchanging the values of $\gamma_1$ and $\gamma_2$ corresponds to reversing the orientation of the third component.
\end{thm}

\begin{proof} We recall the following facts about the involution:

\begin{enumerate} [label=($A$-\arabic*)]
\item \label{iota:chainmap} The map $\iota_L$ is a chain map.
\item \label{iota:reversing} The map $\iota_L$ sends an element in Alexander grading $(A_1, A_2, A_3) = (i,j,k)$ to an element in Alexander grading $(A_1, A_2, A_3) = (-i,-j,-k)$.
\item \label{iota:square} The map $\iota_L$ has square homotopic to the Sarkar map $\xi = \xi_1 \circ \xi_2 \circ \xi_3$.
\end{enumerate}

We begin with $F_0$. We observe that the elements $11z, 11w, 1z1, 1w1, z11, w11$ are the only terms in the span of the generators in their respective Alexander triple gradings which have differential zero. By \ref{iota:chainmap} and \ref{iota:reversing}, we have that $\iota_L$ interchanges these generators in pairs, to wit $11z \leftrightarrow 11w$, $1z1 \leftrightarrow 1w1$, and $z11 \leftrightarrow w11$. This fixes the involution on filtration level $F_0$ to be given by the bar map.

We continue now to $F_1$. Applying the computation of the involution on $F_0$ and the facts \ref{iota:chainmap} and \ref{iota:reversing} additionally fixes the action on $F_1$. For the generators in $F_1$ of the form $abc$ such that exactly one of $a,b,c$ is $1$ and the other two are elements of $\{z,w\}$, we see that the involution $\iota_L$ is given by the bar map. Special attention is due to the remaining five generators of $F_1$, which lie in grading $(0,0,0)$. These are the generators $\{1 1 \theta, 1 \theta 1, \theta 1 1, \ve{X}, \ve{Y}\}$. Analysis of the relationship \ref{iota:chainmap} shows that the behavior of these elements must be as follows:
\begin{itemize}
\item The three elements $11 \theta, 1 \theta 1, \theta 11$ are fixed; that is, $\iota_L$ acts on these elements by the bar map.
\item $\ve{X} \mapsto \ve{X} + \theta 11 + 11\theta$
\item $\ve{Y} \mapsto \ve{Y} + 1 \theta 1 + 11 \theta$.
\end{itemize}

Armed with this knowledge we may proceed to $F_2$. The eight elements of the form $abc$ such that $a,b,c \in \{z,w\}$ are the only generators in their respective Alexander gradings and by \ref{iota:square} are exchanged in pairs by the bar map. Now we consider elements of the form $abc$ for which one of $a,b,c$ is $\theta$, one is an element of $\{z,w\}$, and one is $1$. Given an element $abc$ of this form, we observe that properties \ref{iota:reversing} and \ref{iota:chainmap} allow $abc$ to map either to $\bar{a}\bar{b}\bar{c}$ or to the sum of the element $\bar{a}\bar{b}\bar{c}$ with the unique generator of $F_0$ sharing its Alexander grading. Property \ref{iota:square} ensures that exactly one of $abc$ and $\bar{a}\bar{b}\bar{c}$ picks up this extra term. In particular, we have the following options, where below the elements $a_1$ and $a_2$ always sum to $1$ modulo two, and so on for other such pairs.

\begin{itemize}

\item $\theta z 1 \mapsto \theta w 1 + a_1 1w1$
\item $\theta w 1 \mapsto \theta z 1 + a_2 1z1$ \\

\item $\theta 1 z \mapsto \theta 1 w + b_1 11w$
\item $\theta 1 w \mapsto \theta 1 z + b_2 11z$ \\

\item $1 \theta z \mapsto 1 \theta w + c_1 11w$
\item $1 \theta w \mapsto 1 \theta z + c_2 11z$ \\

\item $z \theta 1 \mapsto w\theta 1 + d_1 w11$
\item $w \theta 1 \mapsto z \theta 1 + d_2 z11$ \\

\item $z 1 \theta \mapsto w 1 \theta + e_1 w11$
\item $w 1 \theta \mapsto z 1 \theta + e_2 z11$ \\

\item $1 z \theta \mapsto 1 w \theta + f_1 1w1$
\item $1 w \theta \mapsto 1 z \theta + f_2 1z1$
\end{itemize}

We temporarily leave the indeterminacy above and proceed to consider $F_3$. In the case of triples $abc$ for which one of $a,b,c$ is $\theta$ and the remaining two are elements of $\{z,w\}$, we have a similar situation to the previous set of generators: such a triple may map to $\bar{a}\bar{b}\bar{c}$ or to the sum of $\bar{a}\bar{b}\bar{c}$ and the unique generator in $F_1$ in the same Alexander grading. In particular we once again have pairs of maps as follows, where once again the sum $g_1+g_2$ is one modulo two, and similarly for the other such pairs.

\begin{itemize}
\item $\theta zz \mapsto \theta ww + g_1 1 ww$
\item $\theta ww \mapsto \theta zz + g_2 1 zz$ \\

\item $\theta wz \mapsto \theta zw + h_1 1zw $
\item $\theta zw \mapsto \theta wz + h_2 1wz$\\

\item $z \theta z \mapsto w \theta w + i_1 w1w $ 
\item $w \theta w \mapsto z \theta z + i_2 z1z$ \\

\item $z \theta w \mapsto w \theta z + j_1 w1z$
\item $w \theta z \mapsto z \theta w + j_2 z1w$\\

\item $zz \theta \mapsto ww\theta + k_1 ww1$
\item $ww \theta \mapsto zz \theta + k_2 zz1$\\ 

\item $wz \theta \mapsto zw \theta + \ell_1 zw1$
\item $zw \theta \mapsto wz \theta + \ell_2 wz1$
\end{itemize}

We pause to examine the implications of \ref{iota:chainmap} for the variables thus far introduced. From the chain map relation applied to $\theta zz$, we learn that $g_1=b_1=a_2$, and correspondingly $g_2=b_2=a_1$. From the chain map relation applied to $\theta wz$, we learn that $h_1=b_1=a_2$ and correspondingly $h_2=b_2=a_1$. So in total $h_1=g_1=b_1=a_2$.

From the chain map relation applied to $z\theta z$, we obtain $i_1=d_1=c_2$, and correspondingly $i_2=d_2=c_1$. From the chain map relation applied to $z \theta w$, we obtain $j_1=d_1=c_2$, and correspondingly $j_2=d_2=c_1$. So in total $i_1=j_1=d_1=c_2$.

Finally, from the chain map relation applied to $zz \theta$, we obtain $k_1=f_1=e_2$, and correspondingly $k_2=f_2=e_1$. From the chain map relation applied to $wz\theta$, we obtain $\ell_1=f_1=e_2$, and correspondingly $\ell_2=f_2=e_1$. So we have $\ell_1=k_1=f_1=e_2$.

In particular, choosing a value for $\alpha_1 = h_2=g_2=b_2=a_1$ and $\alpha_2=h_1=g_1=b_1=a_2$ so that $\alpha_1+\alpha_2=1$ modulo two,  for $\beta_1 = i_2=j_2=d_2=c_1$ and $\beta_2=i_1=j_1=d_1=c_2$ so that $\beta_1+\beta_2=1$ modulo two, and finally for $\gamma_1 = \ell_2=k_2=f_2=e_1$ and $\gamma_2=\ell_1=k_1=f_1=e_2$ so that $\gamma_1+\gamma_2$ modulo two is enough to determine all the maps above.

We now turn our attention to the remaining three elements in $F_3$, which are $\theta \theta 1, \theta 1 \theta$, and $1\theta \theta$. We first notice that the chain map relation \ref{iota:chainmap} applied to these three elements shows that each of these three elements is sent to a sum of itself with generators of $F_1$ in grading $(0,0,0)$. We analyze $\theta 1 \theta$ carefully; the other two generators will admit a similar argument. Suppose that $\iota_L$ takes the following form on this generator, where all of the variables $\eta_i$ are either zero or one.
\[ \theta 1 \theta  \mapsto \theta 1 \theta + \eta_1 11 \theta +\eta_2 1 \theta 1 + \eta_3 \theta 11 + \eta_4 \ve{X} +\eta_5 \ve{Y}\]
The condition \ref{iota:square} that $\iota_L$ squares to the Sarkar map implies that $\eta_4=1$ and $\eta_5=0$. Moreover, the condition \ref{iota:chainmap} that $\iota_L$ is a chain map implies that $\eta_2=0$, $\eta_1=b_2 = \alpha_1$, and $\eta_3=e_2=\gamma_2$. This results in
\[ \theta 1 \theta  \mapsto \theta 1 \theta + \alpha_1 11 \theta + \gamma_2 \theta 11 + \ve{X} .\]
Applying similar logic to the other two generators we conclude that
\begin{align*}
\theta \theta 1 &\mapsto \theta \theta 1 + \alpha_2 1 \theta 1 + \beta_1 \theta 1 1 + \ve{X} + \ve{Y}\\
\theta 1 \theta  &\mapsto \theta 1 \theta + \alpha_1 11 \theta + \gamma_2 \theta 11 + \ve{X} \\
1 \theta \theta &\mapsto  1 \theta \theta + \beta_2 11 \theta + \gamma_1 1\theta 1 + \ve{Y}.
\end{align*}

We are now ready to examine the six generators of $F_4$, which are triples $abc$ such that two of $a,b,c$ are equal to $\theta$ and the remaining term is an element of $\{z,w\}$. We notice that the relations \ref{iota:chainmap} and \ref{iota:reversing} imply that a triple $abc$ in $F_4$ is sent to a a sum of $\bar{a}\bar{b}\bar{c}$ with elements of lower filtration level in the same Alexander grading. We examine the pair $\theta \theta z$ and $\theta \theta w$; a similar analysis applies to the other two pairs. Suppose that we have
\begin{align*}
\theta \theta z &\mapsto \theta \theta w + \delta_1 1 \theta w + \delta_2 \theta 1 w + \delta_3 11 w \\
\theta \theta w &\mapsto \theta \theta z + \epsilon_1 1 \theta z + \epsilon_2 \theta 1 z + \epsilon_3 11z.
\end{align*}
From \ref{iota:square}, we obtain $\epsilon_1=\delta_1+1$ and $\epsilon_2=\delta_2+1$ modulo two, furthermore, we find that $\delta_3 + \epsilon_3 = \delta_1 c_2 + \delta_2b_2 +1 = \epsilon_1 c_1 + \epsilon_2 b_1 +1$ modulo two. From \ref{iota:chainmap}, we obtain $\delta_1=\alpha_2$ and $\delta_2 = \beta_1$. The pair now reduces to
\begin{align*}
\theta \theta z &\mapsto \theta \theta w + \alpha_2 1 \theta w + \beta_1 \theta 1 w + \delta_3 11 w \\
\theta \theta w &\mapsto \theta \theta z + \alpha_1 1 \theta z + \beta_2 \theta 1 z + \epsilon_3 11z
\end{align*}
\noindent where $\delta_3+\epsilon_3 = \alpha_1\beta_1 + \alpha_2\beta_2 +1$. Up to possibly changing basis to replace $\theta \theta w$ with $\theta \theta w + 11w$, we may assume $\delta_3 =1$, so that we have
\begin{align*}
\theta \theta z &\mapsto \theta \theta w + \alpha_2 1 \theta w + \beta_1 \theta 1 w + 11 w \\
\theta \theta w &\mapsto \theta \theta z + \alpha_1 1 \theta z + \beta_2 \theta 1 z + (\alpha_1 \beta_1 + \alpha_2 \beta_2) 11z.
\end{align*}

Applying the same analysis to the remaining four generators gives us
\begin{align*}
\theta \theta z &\mapsto \theta \theta w + \alpha_2 1 \theta w + \beta_1 \theta 1 w + 11 w \\
\theta \theta w &\mapsto \theta \theta z + \alpha_1 1 \theta z + \beta_2 \theta 1 z + (\alpha_1 \beta_1 + \alpha_2 \beta_2) 11z \\
\theta z \theta &\mapsto \theta w \theta + \gamma_1 \theta w 1 + \alpha_2 1w\theta + 1w1 \\
\theta w \theta & \mapsto \theta z \theta + \gamma_2 \theta z 1 + \alpha_1 1z\theta + (\alpha_1 \gamma_1 + \alpha_2\gamma_2) 1z1 \\
z \theta \theta & \mapsto w\theta \theta + \beta_1 w1\theta + \gamma_1 w \theta 1 + w11\\
w \theta \theta &\mapsto z \theta \theta + \beta_2 z1\theta + \gamma_2 z \theta 1 + (\beta_1\gamma_1 + \beta_2\gamma_2)z11 
\end{align*}
again up to a change of basis.

This gives the eight claimed possibilites for $\iota_L$, which it is straightforward to confirm are nonconjugate. The statement about orientation reversal follows by applying $\xi_i$ to $\iota_L$ for $i=1,2,3$.
\end{proof}

\begin{rem} As a final note, we speculate on the relationship between the eight possible maps $\iota_L$ described above. We label each possibile map by the set $\{\alpha_i, \beta_j, \gamma_k\}$ of three variables which must be set equal to one to achieve it. We observe that there is a chain map on the complex $\mathcal B$ which sends a triple $abc$ to $bca$, rotates the elements $\{\ve{X}, \ve{Y}, \ve{X}+\ve{Y}\}$ cyclically so that $\ve{X}$ is taken to $\ve{Y}$ and so on, and furthermore similarly rotates the variables $\{W_3, W_2, W_1\}$ and $\{Z_3, Z_2, Z_1\}$ cyclically. We observe that applying this chain map rotates through each of the triple of maps
\[ \{\{\alpha_1, \beta_1, \gamma_2\}, \{\alpha_1, \beta_2, \gamma_1\}, \{\alpha_2, \beta_1, \gamma_1\}\} \]
\[ \{\{\alpha_2, \beta_2, \gamma_1\}, \{\alpha_2, \beta_1, \gamma_2\}, \{\alpha_1, \beta_2, \gamma_2\}\}\]
and fixes each of the maps $\{\alpha_1, \beta_1, \gamma_1\}$ and $\{\alpha_2,\beta_2,\gamma_2\}$. We conjecture that these sets of maps correspond to the orbits of the set of eight orientations on the Borromean rings produced by rotation.
\end{rem}

\begin{rem} We speculate that working over $\Z$, as in \cite{AM:HFZ}, one might be able to disambiguate which of the above maps corresponds to which choice of twisting directions $\frd$. Over $\Z$, the Sarkar maps for different twisting directions are not generally chain homotopic, so each map $\iota_{L,\frd}$ would square to a distinct map over $\Z$, and therefore could potentially be disambiguated. 
\end{rem}

\bibliographystyle{custom}
\def\MR#1{}
\bibliography{biblio}

% \bib, bibdiv, biblist are defined by the amsrefs package.
\begin{bibdiv}
\begin{biblist}

\bib{AM:HFZ}{misc}{
      author={Abouzaid, Mohammed},
      author={Manolescu, Ciprian},
       title={Canonical orientation in {H}eegaard {F}loer homology},
        date={2025},
        note={arXiv:2510.20062},
}

\bib{BLZLattice}{article}{
      author={Borodzik, Maciej},
      author={Liu, Beibei},
      author={Zemke, Ian},
       title={Lattice homology, formality, and plumbed {L}-space links},
        date={2024},
     journal={J. Eur. Math. Soc.},
        note={Published Online first. arXiv:2210.15792.},
}

\bib{Crainic}{misc}{
      author={Crainic, Marius},
       title={On the perturbation lemma, and deformations},
        date={2004},
        note={arXiv:0403266},
}

\bib{CZZ}{unpublished}{
      author={Chen, Daren},
      author={Zemke, Ian},
      author={Zhou, Hugo},
       title={L-space satellite operators and knot {F}loer homology},
        date={2024},
        note={arXiv:2412.05755},
}

\bib{GLLM_Borromean}{article}{
      author={Gorsky, Eugene},
      author={Lidman, Tye},
      author={Liu, Beibei},
      author={Moore, Allison~H.},
       title={Triple linking numbers and {H}eegaard {F}loer homology},
        date={2023},
        ISSN={1073-7928,1687-0247},
     journal={Int. Math. Res. Not. IMRN},
      number={6},
       pages={4501\ndash 4554},
         url={https://doi-org.uoregon.idm.oclc.org/10.1093/imrn/rnab368},
      review={\MR{4565671}},
}

\bib{GorskyNemethiLattice}{article}{
      author={Gorsky, Eugene},
      author={N\'{e}methi, Andr\'{a}s},
       title={Lattice and {H}eegaard {F}loer homologies of algebraic links},
        date={2015},
        ISSN={1073-7928},
     journal={Int. Math. Res. Not. IMRN},
      number={23},
       pages={12737\ndash 12780},
         url={https://doi.org/10.1093/imrn/rnv075},
      review={\MR{3431635}},
}

\bib{GorskyNemethiAlgebraicLinks}{article}{
      author={Gorsky, Eugene},
      author={N\'{e}methi, Andr\'{a}s},
       title={Links of plane curve singularities are {$L$}-space links},
        date={2016},
        ISSN={1472-2747},
     journal={Algebr. Geom. Topol.},
      volume={16},
      number={4},
       pages={1905\ndash 1912},
         url={https://doi.org/10.2140/agt.2016.16.1905},
      review={\MR{3546454}},
}

\bib{Mac2}{misc}{
      author={Grayson, Daniel~R.},
      author={Stillman, Michael~E.},
       title={Macaulay2, a software system for research in algebraic geometry},
         how={Available at \texttt{http://www.math.uiuc.edu/Macaulay2/}},
}

\bib{HHSZNaturality}{article}{
      author={Hendricks, Kristen},
      author={Hom, Jennifer},
      author={Stoffregen, Matthew},
      author={Zemke, Ian},
       title={Naturality and functoriality in involutive {H}eegaard {F}loer
  homology},
        date={2026},
        ISSN={1663-487X,1664-073X},
     journal={Quantum Topol.},
      volume={17},
      number={1},
       pages={45\ndash 188},
         url={https://doi.org/10.4171/qt/210},
      review={\MR{5032324}},
}

\bib{HMInvolutive}{article}{
      author={Hendricks, Kristen},
      author={Manolescu, Ciprian},
       title={Involutive {H}eegaard {F}loer homology},
        date={2017},
     journal={Duke Math. J.},
      volume={166},
      number={7},
       pages={1211\ndash 1299},
}

\bib{JTNaturality}{article}{
      author={Juh\'asz, Andr\'as},
      author={Thurston, Dylan},
      author={Zemke, Ian},
       title={Naturality and mapping class groups in {H}eegard {F}loer
  homology},
        date={2021},
        ISSN={0065-9266,1947-6221},
     journal={Mem. Amer. Math. Soc.},
      volume={273},
      number={1338},
       pages={v+174},
         url={https://doi.org/10.1090/memo/1338},
      review={\MR{4337438}},
}

\bib{LiuLSpaceLinks}{article}{
      author={Liu, Yajing},
       title={{$L$}-space surgeries on links},
        date={2017},
        ISSN={1663-487X,1664-073X},
     journal={Quantum Topol.},
      volume={8},
      number={3},
       pages={505\ndash 570},
         url={https://doi.org/10.4171/QT/96},
      review={\MR{3692910}},
}

\bib{LOTBordered}{article}{
      author={Lipshitz, Robert},
      author={Ozsvath, Peter~S.},
      author={Thurston, Dylan~P.},
       title={Bordered {H}eegaard {F}loer homology},
        date={2018},
        ISSN={0065-9266},
     journal={Mem. Amer. Math. Soc.},
      volume={254},
      number={1216},
       pages={viii+279},
         url={https://doi.org/10.1090/memo/1216},
      review={\MR{3827056}},
}

\bib{OSAlternating}{article}{
      author={Ozsv\'ath, Peter},
      author={Szab\'o, Zolt\'an},
       title={Heegaard {F}loer homology and alternating knots},
        date={2003},
        ISSN={1465-3060,1364-0380},
     journal={Geom. Topol.},
      volume={7},
       pages={225\ndash 254},
         url={https://doi.org/10.2140/gt.2003.7.225},
      review={\MR{1988285}},
}

\bib{OSKnots}{article}{
      author={Ozsv\'ath, Peter},
      author={Szab\'o, Zolt\'an},
       title={Holomorphic disks and knot invariants},
        date={2004},
     journal={Adv. Math.},
      volume={186},
      number={1},
       pages={58\ndash 116},
}

\bib{OSLens}{article}{
      author={Ozsv{\'a}th, Peter},
      author={Szab{\'o}, Zolt{\'a}n},
       title={On knot {F}loer homology and lens space surgeries},
        date={2005},
        ISSN={0040-9383},
     journal={Topology},
      volume={44},
      number={6},
       pages={1281\ndash 1300},
         url={https://doi.org/10.1016/j.top.2005.05.001},
}

\bib{OSLinks}{article}{
      author={Ozsv\'ath, Peter},
      author={Szab\'o, Zolt\'an},
       title={Holomorphic disks, link invariants and the multi-variable
  {A}lexander polynomial},
        date={2008},
     journal={Algebr. Geom. Topol.},
      volume={8},
      number={2},
       pages={615\ndash 692},
}

\bib{RasmussenKnots}{thesis}{
      author={Rasmussen, Jacob},
       title={{F}loer homology and knot complements},
        type={Ph.D. Thesis},
        date={2003},
        note={arXiv:math/0306378},
}

\bib{SarkarMovingBasepoints}{article}{
      author={Sarkar, Sucharit},
       title={Moving basepoints and the induced automorphisms of link {F}loer
  homology},
        date={2015},
     journal={Algebr. Geom. Topol.},
      volume={15},
      number={5},
       pages={2479\ndash 2515},
}

\bib{Seidel_HMS_Quartic}{article}{
      author={Seidel, Paul},
       title={Homological mirror symmetry for the quartic surface},
        date={2015},
        ISSN={0065-9266,1947-6221},
     journal={Mem. Amer. Math. Soc.},
      volume={236},
      number={1116},
       pages={vi+129},
         url={https://doi-org.uoregon.idm.oclc.org/10.1090/memo/1116},
      review={\MR{3364859}},
}

\bib{Weibel}{book}{
      author={Weibel, Charles~A.},
       title={An introduction to homological algebra},
      series={Cambridge Studies in Advanced Mathematics},
   publisher={Cambridge University Press, Cambridge},
        date={1994},
      volume={38},
        ISBN={0-521-43500-5; 0-521-55987-1},
         url={https://doi.org/10.1017/CBO9781139644136},
      review={\MR{1269324}},
}

\bib{ZemQuasi}{article}{
      author={Zemke, Ian},
       title={Quasistabilization and basepoint moving maps in link {F}loer
  homology},
        date={2017},
        ISSN={1472-2747},
     journal={Algebr. Geom. Topol.},
      volume={17},
      number={6},
       pages={3461\ndash 3518},
         url={https://doi.org/10.2140/agt.2017.17.3461},
      review={\MR{3709653}},
}

\end{biblist}
\end{bibdiv}

\end{document}